\documentclass[11pt]{article}
\usepackage{amsmath, amssymb, amsfonts, amsthm,latexsym}
\usepackage{graphicx}

\newtheorem*{urem}{Remark}

\newtheorem{theorem}{Theorem}[section]
\newtheorem{claim}[theorem]{Claim}
\newtheorem{rem}[theorem]{Remark}
\newtheorem{lemma}[theorem]{Lemma}
\newtheorem{propos}[theorem]{Proposition}
\newtheorem{cor}[theorem]{Corollary}

\newcommand{\eps}{\varepsilon}
\DeclareMathOperator{\sgn}{sgn}
\newcommand{\M}{\mathcal{M}}

\title{A differential version of the Chebyshev-Markov-Stieltjes inequalities}
\author{Shoni Gilboa\thanks{Mathematics Dept., The Open University of Israel, Raanana 43107, Israel, Email: tipshoni@gmail.com.}
\and Ron Peled\thanks{Mathematics Dept., Tel-Aviv University,
Tel-Aviv 69978, Israel, Email: peledron@post.tau.ac.il. Supported by
an ISF grant and an IRG grant.}}

\begin{document}
\maketitle
\begin{abstract}
We show that a differential version of the classical
Chebyshev-Markov-Stieltjes inequalities holds for a broad family of
weight functions. Such a differential version appears to be new. Our
results apply to weight functions which are bounded away from zero
and piecewise absolutely continuous and yield effective estimates
when the weight satisfies additional regularity conditions.
\end{abstract}

\section{Introduction}

Let $w$ be a non-negative, integrable function on the interval
$[-1,1]$, with non-zero integral. Let $n\ge 1$ and let $\M_n$ be the
class of all positive Borel measures $\mu$ on $[-1,1]$ satisfying
that
\begin{equation*}
  \int_{-1}^1 p(t) w(t)dt = \int p(t) d\mu(t)
\end{equation*}
for all polynomials $p$ of degree at most $n$. Define the extremal
functions
\begin{equation}\label{eq:pi_underline_pi_lambda_def}
  \pi(x) := \sup_{\mu\in \M_n} \mu([-1,x]),\quad \underline{\pi}(x) := \inf_{\mu\in \M_n}
  \mu([-1,x)),\quad \lambda(x) := \sup_{\mu\in \M_n} \mu(\{x\}).
\end{equation}
Figures~\ref{fig:pi_integral_and_pi_underline} and \ref{fig:lambda}
depict these functions for a specific choice of weight function.
Trivially,
\begin{equation}\label{eq:Chebyshev_Markov_Stieltjes_intro}
\underline{\pi}(x)\leq\int_{-1}^{x}w(t)dt\leq\pi(x),\quad -1\le x\le
1.
\end{equation}
Investigations going back to the work of Chebyshev, Markov and
Stieltjes have shown that for each $-1\le x\le 1$ there is a unique
measure simultaneously attaining the suprema and infimum in
\eqref{eq:pi_underline_pi_lambda_def}. These extremal measures form
an explicit one-parameter family of atomic measures in $\M_n$,
termed the \emph{canonical representations}. In
Section~\ref{sec:background} we review the theory of canonical
representations. The explicit identification of the extremal
measures makes the inequalities
\eqref{eq:Chebyshev_Markov_Stieltjes_intro} into a powerful tool,
termed the Chebyshev-Markov-Stieltjes inequalities. One consequence
of the fact that the extrema in
\eqref{eq:pi_underline_pi_lambda_def} are attained on the same
measure is the following relation between the extremal functions,
\begin{equation*}
  \pi(x)-\underline{\pi}(x)=\lambda(x).
\end{equation*}
This implies, together with
\eqref{eq:Chebyshev_Markov_Stieltjes_intro}, that
\begin{equation}\label{eq:Chebyshev_Markov_Stieltjes_with_lambda_intro}
\left|\pi(x) - \int_{-1}^{x}w(t)dt\right|\le\lambda(x)\quad\text{
and }\quad \left|\underline{\pi}(x) -
\int_{-1}^{x}w(t)dt\right|\leq\lambda(x),\quad -1\le x\le 1.
\end{equation}
In this paper we show that the inequalities
\eqref{eq:Chebyshev_Markov_Stieltjes_with_lambda_intro} may be valid
also in a pointwise rather than cumulative sense. We focus on the
class of weight functions which are bounded away from zero and
satisfy certain regularity conditions.

Observe that $\pi$ is a non-decreasing function and thus
differentiable almost everywhere on $(-1,1)$. In fact, more is true,
the function $\pi$ is analytic at all but finitely many points of
$(-1,1)$, where the exceptional points are roots of the principal
representations, as elaborated in Section~\ref{sec:background}.

\begin{theorem}\label{thm:Lipschitz}
Suppose $w$ is a function on $[-1,1]$ satisfying that for some
$m,R>0$:
\begin{itemize}
\item  $w(x)\geq m$ for every $-1\leq x\leq 1$.
\item  $|w(x)-w(y)|\leq R\cdot(y-x)$ for every $-1\leq x<y\leq 1$.
\end{itemize}
Then there exists a constant $C(w)>0$ such that for every
differentiability point $x\in(-1,1)$ of $\pi$,
\begin{equation}\label{eq:required_pi_minus_estimate1}
-C(w)\frac{\lambda(x)}{1-x}\le \pi'(x)-w(x)\leq
C(w)\lambda(x)\min\left\{\frac{1}{1+x},n^2\right\}.
\end{equation}
\end{theorem}
This inequality appears to be new. One motivation for its
development came from papers of Kuijlaars
\cite{Kuijlaars1,Kuijlaars2} where a lower bound for $\pi'$ for the
case of Jacobi weight functions played a central role.
Figure~\ref{fig:pi_derivative} presents a plot of the function $\pi'
- w$. The constant $C(w)$ appearing in the inequality may be given
an explicit estimate depending only on $m$ and $R$, yielding
uniformity of our estimates for the class of weight functions
satisfying the assumptions of the theorem, see the discussion in
Section~\ref{sec:open_problems}.

It is worth noting that when $w$ is bounded above, the function
$\lambda(x)$ cannot be too large. Specifically, it satisfies
\begin{equation*}
\lambda(x)\leq \frac{C
M}{n}\max\left\{\sqrt{1-x^2},\frac{1}{n}\right\},\quad -1<x<1,
\end{equation*}
where $C>0$ is an absolute constant and $M$ is the maximum of $w$ on
$[-1,1]$ (see Lemma~\ref{lambda} below). Thus the above theorem may
also be seen as a practical tool for estimating the density $w$ if
one is given only the first few moments of the measure $w(t)dt$.
With regards to this we mention that Lemma~\ref{deriv-shift} gives
an explicit expression for $\pi'$.

We also remark that in Theorem~\ref{thm:Lipschitz}, as well as the
next two theorems, one immediately obtains corresponding bounds with
$\underline{\pi}$ replacing $\pi$ by considering the reversed weight
function $\tilde{w}(x) := w(-x)$, as we have the identity
$\pi^{\tilde{w}}(x) = \int_{-1}^1 w(t)dt - \underline{\pi}^w(-x)$.

The next two theorems provide estimates for the difference $\pi' -
w$ under weaker regularity assumptions on $w$. In the first of these
the assumption of Lipschitz continuity is relaxed to a Sobolev
condition. In the second, the regularity assumptions on $w$ are
relaxed to mere piecewise absolute continuity, thus allowing a class
of discontinuous weight functions, at the price of making the result
non-quantitative.

\begin{figure}[t]
\centering
{\includegraphics[width=0.7\textwidth]{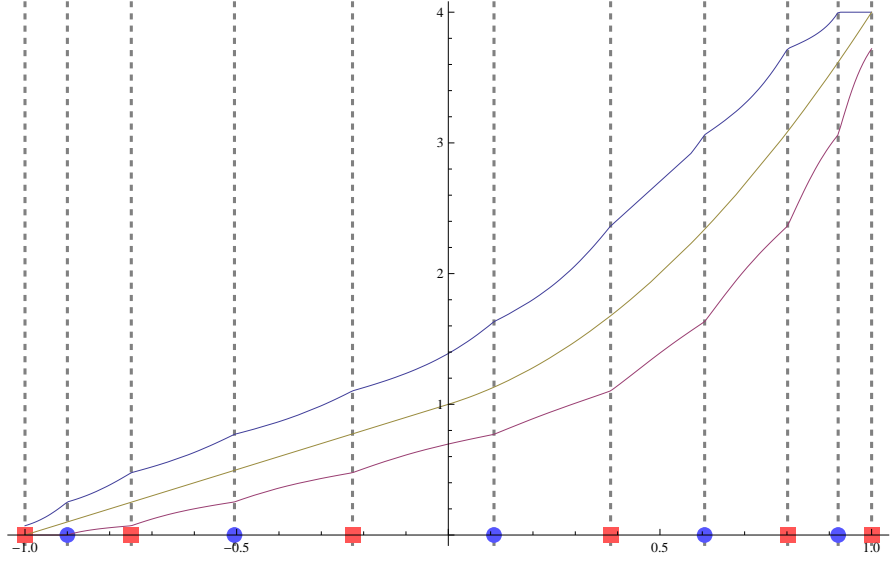}}
{\caption{A plot of $\pi$ (top graph), $\int_{-1}^x w(t)dt$ (middle
graph) and $\underline{\pi}$ (bottom graph) for $n=5$ and the weight
function $w(t)=\max\{1,1+4t\}$. The circles on the axis denote the
nodes of the Gaussian quadrature and the squares denote the nodes of
the Lobatto quadrature. These quadratures are defined in
Section~\ref{sec:background}, where the fact that $\pi$ is constant
to the right of the last Gaussian node is also
explained.\label{fig:pi_integral_and_pi_underline}}}
\end{figure}

\begin{figure}[t]
\centering
{\includegraphics[width=0.7\textwidth]{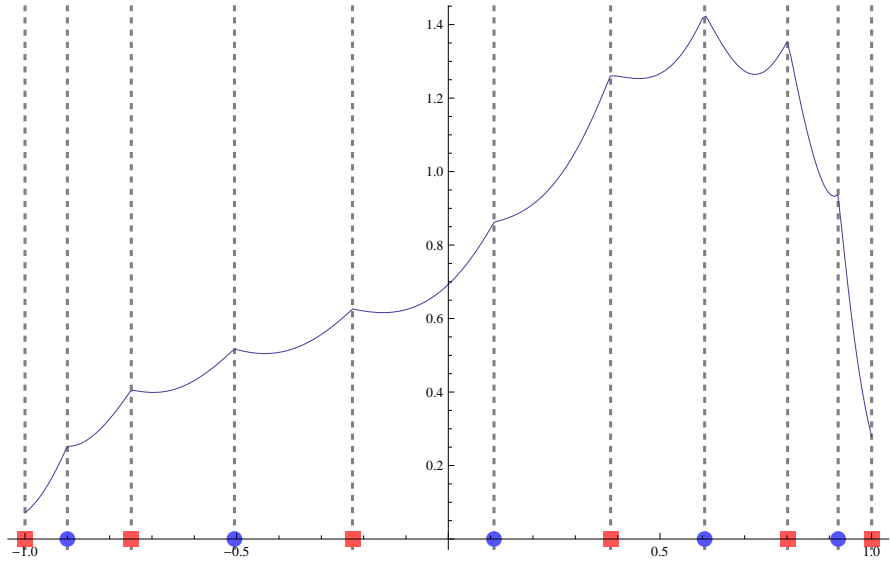}}
{\caption{A plot of the function $\lambda$ for $n=5$ and the weight
function $w(t)=\max\{1,1+4t\}$. The circles on the axis denote the
nodes of the Gaussian quadrature and the squares denote the nodes of
the Lobatto quadrature as defined in Section~\ref{sec:background}.
\label{fig:lambda}}}
\end{figure}

\begin{figure}[t]
\centering
{\includegraphics[width=0.7\textwidth]{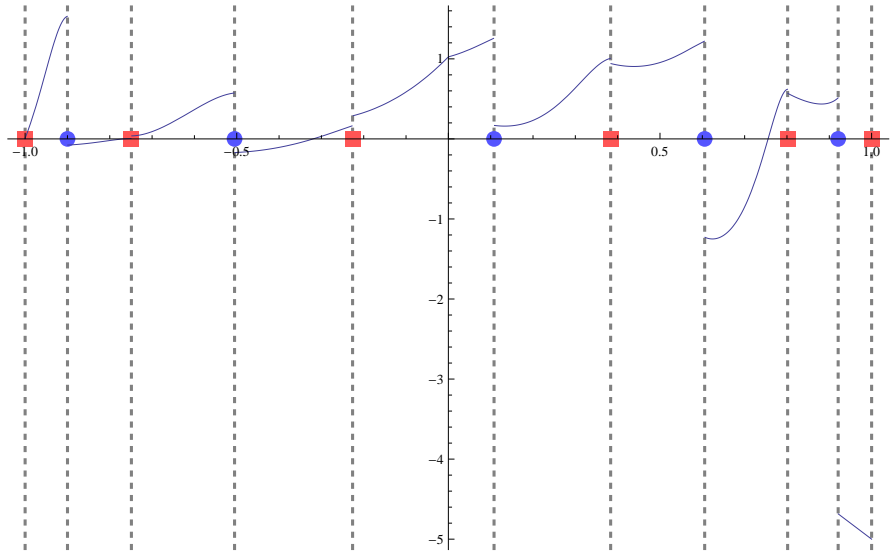}}
{\caption{A plot of the function $\pi' - w$ for $n=5$ and the weight
function $w(t)=\max\{1,1+4t\}$. The circles on the axis denote the
nodes of the Gaussian quadrature and the squares denote the nodes of
the Lobatto quadrature as defined in Section~\ref{sec:background}.
The lack of smoothness at $t=0$ is due to the lack of smoothness of
$w$ at this point as we know that $\pi$ is analytic there, see
Section~\ref{sec:background}. \label{fig:pi_derivative}}}
\end{figure}

\begin{theorem}\label{thm:abs_cont}
Suppose $w$ is a function on $[-1,1]$ satisfying that for some $m>0$
and $p>1$:
\begin{itemize}
  \item $w(x)\geq m$ for every $-1\leq x\leq 1$.
  \item $w$ is absolutely continuous and $w'\in
L_p[-1,1]$.
\end{itemize}
Then there exists a constant $C(w)>0$ such that for every
differentiability point $x\in(-1,1)$ of $\pi$,
\begin{equation}\label{eq:required_pi_minus_estimate3}
-C(w)\left(\frac{1}{1-x}+\frac{1}{\lambda(x)^{1/p}}\right)\lambda(x)\le
\pi'(x)-w(x)\leq
C(w)\left(\min\left\{\frac{1}{1+x},n^2\right\}+\frac{1}{\lambda(x)^{1/p}}\right)\lambda(x).
\end{equation}
\end{theorem}

\begin{theorem}\label{thm:discont}
Suppose $w$ is a function on $[-1,1]$ satisfying that for some $m>0$
and $-1<s_1<s_2<\ldots<s_L<1$:
\begin{itemize}
  \item $w(x)\geq m$ for every $-1\leq x\leq 1$.
  \item $w$ is absolutely continuous on each of the intervals
  $[-1,s_1),(s_1,s_2),\ldots,(s_L,1]$.
\end{itemize}
Then for every $\eps>0$ there exists a $C(w,\eps)>0$ such that
\begin{align}
|\pi'(x)-w(x)|\le \eps
\end{align}
for every differentiability point $x\in(-1,1)$ of $\pi$ satisfying
$1-x^2\ge \frac{C(w,\eps)}{n^2}$ and $|s_i -
x|\ge\frac{C(w,\eps)}{n}$ for $1\le i\le L$.
\end{theorem}
\begin{urem}
The absolute continuity assumption in Theorem~\ref{thm:discont}
amounts to saying that $w$ is the sum of an absolutely continuous
function on $[-1,1]$ and a step function with discontinuities only
at the points $s_1,s_2,\ldots,s_L$. In particular, $w$ has left and
right limits at each of the $s_i$.

A more refined version of Theorem~\ref{thm:discont} may be obtained
using Remark~\ref{rem:refined_discontinuity}.
\end{urem}

Section~\ref{sec:background} presents the notation and background
used throughout the paper. In
Section~\ref{sec:preliminary_estimates} we gather several estimates
for orthogonal polynomials and quadrature formulas which will be
used throughout the proofs. The estimates of this section are all
essentially known but the concise presentation may be of use to a
non-expert in the literature. Accordingly, we include short proofs
for most of the statements there, relying on various comparison
arguments to the case of the constant weight function. An important
role in these is played by a theorem of Badkov, see
Theorem~\ref{thm:Badkov_thm} below. In
Section~\ref{sec:root_separation} we provide lower bounds on the
distance between nodes of canonical representations. In
Section~\ref{sec:p_a_derivative} we consider the polynomials whose
roots are the nodes of the canonical representations and establish
lower bounds for their derivatives at these nodes. We also prove
there that the interpolation polynomials corresponding to the
canonical representations exhibit strong localization properties.
Theorems~\ref{thm:Lipschitz} and \ref{thm:abs_cont} are proved in
Section~\ref{sec:pi_derivative}. Theorem~\ref{thm:discont} is proved
in Section~\ref{sec:thm_main2_proof}. In the final
Section~\ref{sec:open_problems} we discuss some open questions.

\section{Notation and background}\label{sec:background}

Throughout the paper we use the following notation. We refer the
reader to the book of Karlin and Studden \cite{Karlin-Studden},
especially to section 2 in chapter IV there, for reference to the
facts mentioned in this section.

Let $w$ be a non-negative, integrable function on the interval
$[-1,1]$, with non-zero integral. Such a $w$ will be called a
\emph{weight function}. Let $n$ be a positive integer and let
$\varphi$ be the $n$th-degree orthonormal polynomial, with positive
leading coefficient, with respect to $w$. Let $\psi$ be the
$(n-1)$th-degree orthonormal polynomial, with positive leading
coefficient, with respect to $(1-t^2)w(t)$. Figure~\ref{fig:phi_psi}
depicts these polynomials for $n=5$ and the weight function
$w(t)=\max\{1,1+4t\}$.

\begin{figure}[t]
\centering
{\includegraphics[width=0.7\textwidth]{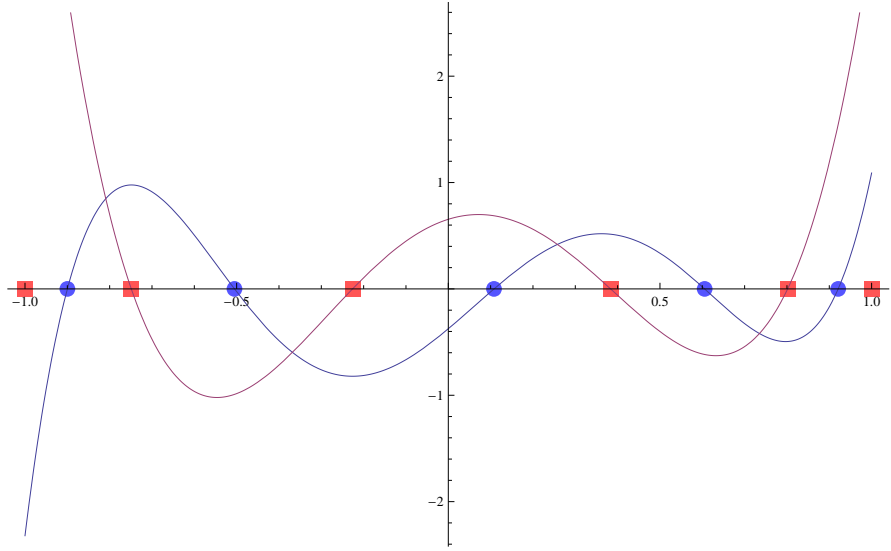}}
{\caption{A plot of the polynomials $\varphi$ and $\psi$ for $n=5$
and the weight function $w(t)=\max\{1,1+4t\}$. The circles on the
axis denote the nodes of the Gaussian quadrature and the squares
denote the nodes of the Lobatto quadrature. \label{fig:phi_psi}} }
\end{figure}

For every real $a$ define the polynomial
\begin{equation}\label{eq:P_a_def}
P_a(t):=\begin{cases}
\varphi(t)-a(1-t)\psi(t) & \mbox{if } a\geq 0 \\
\varphi(t)-a(1+t)\psi(t) & \mbox{if }a\leq 0 \end{cases}.
\end{equation}

$P_a$ has $n$ simple zeros in $(-1,1)$ which we denote by
\begin{equation*}
-1<\xi_1(a)<\ldots<\xi_n(a)<1.
\end{equation*}
Also $(1-t^2)\psi(t)$ has $n+1$ simple roots which we denote by
\begin{equation*}
-1=\eta_0<\eta_1<\ldots<\eta_n=1.
\end{equation*}
We will make use of the fact that
\begin{equation}\label{eq:xi_a_increases}
  \text{as $a$ increases from $-\infty$ to $\infty$, $\xi_i(a)$ increases
from $\eta_{i-1}$ to $\eta_i$}.
\end{equation}
Consistently with this fact, we denote $\xi_i(-\infty):=\eta_{i-1}$,
$\xi_i(\infty):=\eta_i$.

A \emph{quadrature} (formula) of degree $k$ for the weight function
$w$ is a set of points $-1\leq t_1\le t_2\le \ldots\le t_N\leq 1$,
called \emph{nodes}, and a set of \emph{weights}
$w_1,w_2,\ldots,w_N$ such that
\begin{equation}\label{eq:quadrature_def}
\int_{-1}^1p(t)w(t)dt=\sum_{i=1}^Nw_i p(t_i)
\end{equation}
for every polynomial $p$ of degree at most $k$. Define the
\emph{index} function
\begin{equation}\label{eq:index_fcn_def}
  I(u) := \begin{cases}
    2 & -1<u<1\\
    1 & u\in\{-1,1\}
  \end{cases}.
\end{equation}
Define also the index of a quadrature formula to be the sum of the
indices of its nodes.

We now describe all the degree $2n-1$ quadrature formulas for $w$,
having positive weights, whose index is at most $2n+1$. These
formulas are called \emph{canonical representations}. We describe
the set of nodes of these formulas, considering separately 4 cases:
\begin{enumerate}
  \item The roots $(\xi_i(0))$ are the nodes of a quadrature
  formula, known as the \emph{Gaussian quadrature} or the lower
  principal representation (see also Figure~\ref{fig:phi_psi}).
  \item The roots $(\eta_i)$ are the nodes of a quadrature
  formula, known as the \emph{Lobatto quadrature} or the upper principal
  representation (see also Figure~\ref{fig:phi_psi}).
  \item For every $0<a<\infty$ the roots $(\xi_i(a))$ with the additional
  node $-1$ are the nodes of a quadrature formula, called a lower
  canonical representation.
  \item For every $-\infty<a<0$ the roots $(\xi_i(a))$ with the additional
  node $1$ are the nodes of a quadrature formula, called an upper
  canonical representation.
\end{enumerate}
Every $-1<x<1$ is therefore a node of exactly one of these
quadrature formulas, and we denote this quadrature formula by
$\Sigma_x$ and its set of nodes by $S_x$.

It is a classical fact that for every $-1<x<1$, the suprema and
infimum in \eqref{eq:pi_underline_pi_lambda_def} are simultaneously
attained on the quadrature formula $\Sigma_x$ (and for
$x\in\{-1,1\}$ they are attained on the Lobatto quadrature).
Correspondingly, for $-1<x<1$ and $-1\leq u\leq 1$ we denote by
$\lambda_x(u)$ the weight of $u$ in $\Sigma_x$ and, for brevity, we
write $\lambda(x)$ instead of $\lambda_x(x)$. Thus,
$\lambda_x(u)\neq 0$ if and only if $u\in S_x$, and
$\lambda_x(u)=\lambda(u)$ if $u\in S_x-\{-1,1\}$.
Figure~\ref{fig:lambda} shows a plot of $\lambda$ for a certain
choice of weight function. In addition, we have that
\begin{align}\label{eq:pi_def}
\pi(x)=\sum_{u\in
S_x\cap[-1,x]}\lambda_x(u)\quad\text{and}\quad\underline{\pi}(x)&=\sum_{u\in
S_x\cap[-1,x)}\lambda_x(u)
\end{align}
are the weights given to the intervals $[-1,x]$ and $[-1,x)$,
respectively, by the quadrature formula $\Sigma_x$.
Figure~\ref{fig:pi_integral_and_pi_underline} shows a plot of $\pi$
and $\underline{\pi}$ for a certain choice of weight function. It
follows that for $x>\xi_n(0)$ we have $\pi(x) = \int_{-1}^1 w(t)dt$,
a fact which is clearly visible in the figure.

We remark also, as mentioned in the introduction, that the function
$\pi$ (and similarly $\underline{\pi}$) is analytic at all
$x\in(-1,1)$ which are not roots of the Gaussian or the Lobatto
quadratures. Indeed, fix $1\leq i\leq n$ and an interval
$I\subseteq\mathbb{R}-\{0\}$ and let us show that $\pi$ is analytic
on the interval $\{x\colon x=\xi_i(a), a\in I\}$. Assume that
$I\subseteq(0,\infty)$ as the case $I\subseteq(-\infty,0)$ is
treated similarly. Let $p_{i,a}(\cdot)$ be the minimum degree
polynomial satisfying $p_{i,a}(-1)=1$, $p_{i,a}(\xi_j(a))=1$ for
$1\leq j\leq i$ and $p_{i,a}(\xi_j(a))=0$ for $i<j\leq n$. Since
$p_{i,a}$ has degree smaller or equal to $2n-1$ we have
$\pi(\xi_i(a))=\int_{-1}^1 p_{i,a}(t)w(t)dt$ for all $a\in I$ by
\eqref{eq:pi_def}. We conclude that $\pi(\xi_i(a))$ is a rational
function of $\xi_1(a),\ldots,\xi_n(a)$ for $a\in I$. In addition, a
simple use of the implicit function theorem shows that
$\xi_1(a),\ldots,\xi_n(a)$ are analytic functions of $a$ for $a\in
I$. Thus we are done by observing that for $a\in I$,
$a=\frac{\varphi(\xi_i(a))}{(1-\xi_i(a))\psi(\xi_i(a))}$ is a
rational function of $\xi_i(a)$.

Throughout the paper we adopt the following policy regarding
constants. The constants $C(w)$ and $c(w)$ will always denote
positive constants whose value depends only on the function $w$. The
values of these constants will be allowed to change from line to
line, even within the same calculation, with the value of $C(w)$
increasing and the value of $c(w)$ decreasing. We similarly use $C$
and $c$ to denote positive absolute constants whose value may change
from line to line.

\section{Preliminary estimates}\label{sec:preliminary_estimates}
This section collects several estimates, all essentially known, on
the polynomials $\varphi$, $\psi$ and $P_a$, their zeros, and the
associated quadrature formulas.

\subsection{Orthogonal polynomials}
The following theorem, which gives lower and upper bounds on the
orthogonal polynomials $\varphi$ and $\psi$, plays an important role
in our work. The theorem is a corollary of results of Badkov
\cite[Theorem 1.2 and Theorem 1.4]{Badkov} (see also Theorem 13.2
there). The results of Badkov hold in great generality, yielding
estimates for orthogonal polynomials for a broad family of weights,
and are mostly stated in terms of trigonometric orthogonal
polynomials. Here we state a reformulation in terms of algebraic
orthogonal polynomials, of the special case that we need.
\begin{theorem}[Badkov]\label{thm:Badkov_thm}
  Suppose that the weight function $w$ has the form
  \begin{equation*}
    w(x) = h(x) (1-x^2)^\alpha\;\quad \alpha>-1,
  \end{equation*}
  where the function $h$ is assumed to be a measurable function on $[-1,1]$, satisfying
  $0<m\le h\le M$ almost everywhere and the `mean-squared H\"older condition'
  \begin{equation}\label{eq:Badkov's condition}
  \forall~ \delta>0,\quad\int_{-\pi}^{\pi} (h(\cos(\theta+\delta)) - h(\cos \theta))^2
    d\theta \le C\delta
  \end{equation}
  for some constants $M,m,C>0$.
  Then there exist constants $C(w),c(w)>0$
such that for all $-1<x<1$,
  \begin{equation}\label{eq:Badkov_bound}
    c(w)\min\left\{n,\frac{1}{\sqrt{1-x^2}}\right\}^{\alpha+1/2} \le \lvert \varphi(x)\rvert+\sqrt{1-x^2}\lvert \psi(x)\rvert
    \le C(w)\min\left\{n,\frac{1}{\sqrt{1-x^2}}\right\}^{\alpha+1/2}.
  \end{equation}
\end{theorem}
We stress that the assumptions in the above theorem may hold even
for weight functions $w$ with discontinuities in $(-1,1)$. For our
purposes we shall use the theorem when $\alpha\in\{0,1\}$ and when
$h$ is itself a weight function satisfying the assumptions of
Theorem~\ref{thm:Lipschitz}, Theorem~\ref{thm:abs_cont} or
Theorem~\ref{thm:discont}. For instance, if $h$ is absolutely
continuous we may verify condition \eqref{eq:Badkov's condition} by
noting that for every $\delta>0$,
\begin{align*}&\int_{-\pi}^{\pi} (h(\cos(\theta+\delta)) - h(\cos \theta))^2
    d\theta \leq \int_{-\pi}^{\pi}\lvert h(\cos(\theta+\delta)) - h(\cos \theta)\rvert d\theta \max_{-1\leq x\leq 1}h(x)=\\
&=\int_{-\pi}^{\pi}\left\lvert\int_{\cos\theta}^{\cos(\theta+\delta)}h'(t)dt\right\rvert d\theta \max_{-1\leq x\leq 1}h(x)\leq \int_{-\pi}^{\pi}\int_{\min(\cos\theta,\cos(\theta+\delta))}^{\max(\cos\theta,\cos(\theta+\delta))}\lvert h'(t)\rvert dt d\theta \max_{-1\leq x\leq 1}h(x)\leq\\
&\leq\left(2\int_{-1}^1\lvert h'(t)\rvert dt \max_{-1\leq x\leq
1}h(x)\right)\delta,
\end{align*}
where in the last inequality we have changed the order of
integration and used the fact that for each $-1\le t\le 1$, the set
of $\theta$ for which $t$ is between $\cos(\theta)$ and
$\cos(\theta+\delta)$, when viewed on the circle, is contained in
the union of the segments $[\arccos(t)-\delta, \arccos(t)]$ and
$[2\pi - \arccos(t)-\delta, 2\pi - \arccos(t)]$.

The above theorem seems to be the most complicated property of
orthogonal polynomials which we require. When the function $h$
satisfies the assumptions of Theorem~\ref{thm:Lipschitz} the proof of the theorem is simpler,
making use of the Korous comparison theorem \cite[Theorem
7.1.3]{Szego}. For completeness, the proof in this case and an overview of the general case are provided in the Appendix.

The following corollary summarizes the bounds on orthogonal
polynomials which we will need for this work.
\begin{cor}\label{cor:Badkov}
  Suppose $w$ satisfies the conditions of
Theorem~\ref{thm:Badkov_thm} with $\alpha=0$. Then there exist
constants $C(w),c(w)>0$ such that for every $-1<x<1$,
\begin{align}
  &\lvert \varphi(x)\rvert \le
  C(w)\min\left\{n,\frac{1}{\sqrt{1-x^2}}\right\}^{1/2},\label{eq:phi_upper_bound}\\
  &\lvert \psi(x)\rvert \le
  C(w)\min\left\{n,\frac{1}{\sqrt{1-x^2}}\right\}^{3/2}\label{eq:psi_upper_bound},\\
  &\lvert \varphi(x)\rvert+(1-x^2)\min\left\{n,\frac{1}{\sqrt{1-x^2}}\right\}\lvert \psi(x)\rvert
  \ge c(w)\min\left\{n,\frac{1}{\sqrt{1-x^2}}\right\}^{1/2}\label{eq:phi_psi_lower_bound_new}.
\end{align}
\end{cor}

\begin{proof}
The corollary follows directly from Theorem~\ref{thm:Badkov_thm}.
The bound \eqref{eq:phi_upper_bound} follows from the case
$\alpha=0$. The bound \eqref{eq:psi_upper_bound} follows from the
case $\alpha=1$ using the upper bound on $\varphi$ in
\eqref{eq:Badkov_bound}.
For the bound \eqref{eq:phi_psi_lower_bound_new}, first observe that
by the case $\alpha=0$,
\begin{equation}\label{eq:Badkov_lower_bound_third_case_proof}
  \lvert \varphi(x)\rvert+\sqrt{1-x^2}\lvert \psi(x)\rvert
  \ge c(w)\min\left\{n,\frac{1}{\sqrt{1-x^2}}\right\}^{1/2},\quad
  -1<x<1.
\end{equation}
Together with the fact that for every $0<\eps<1$ we have
\begin{equation*}
  \lvert \varphi(x)\rvert+(1-x^2)\min\left\{n,\frac{1}{\sqrt{1-x^2}}\right\}\lvert \psi(x)\rvert
  \ge \sqrt{\eps}\left(\lvert \varphi(x)\rvert+\sqrt{1-x^2}\lvert \psi(x)\rvert\right),\quad
  1-x^2\ge\eps n^{-2},
\end{equation*}
we see that we need only prove \eqref{eq:phi_psi_lower_bound_new}
when $1-x^2< \eps n^{-2}$ for some fixed $0<\eps<1$. To this end,
observe that by \eqref{eq:psi_upper_bound} we have
\begin{equation*}
  \lvert\psi(x)\rvert \le C(w)n^{3/2},\quad -1<x<1.
\end{equation*}
Substituting this relation into
\eqref{eq:Badkov_lower_bound_third_case_proof} yields that for every
$0<\eps<1$,
\begin{equation*}
  \lvert \varphi(x)\rvert
  \ge \left(c(w)-\sqrt{\eps}C(w)\right)n^{1/2},\quad 1-x^2\le \eps n^{-2}.
\end{equation*}
Fixing $\eps$ sufficiently small so that the right-hand side is
positive proves \eqref{eq:phi_psi_lower_bound_new} when $1-x^2< \eps
n^{-2}$.
\end{proof}

The following proposition uses the upper bounds of the previous
corollary to obtain derivative bounds. To obtain unified expressions
we make use of the sign function, defined by
\begin{equation}\label{eq:sgn_a_def}
  \sgn(a):=\begin{cases}
    1&a>0\\
    0&a=0\\
    -1&a<0.
  \end{cases}
\end{equation}
\begin{propos}\label{Bernstein} Suppose $w$ satisfies the conditions of
Theorem~\ref{thm:Badkov_thm} with $\alpha=0$. Then there exists a
constant $C(w)>0$ such that for every $-1<x<1$,
\begin{align}
  |\varphi'(x)|&\leq C(w)
  n\cdot\min\left\{n,\frac{1}{\sqrt{1-x^2}}\right\}^{3/2},\label{eq:phi_deriv_estimate}\\
  |\psi'(x)|&\leq C(w)
  n\cdot\min\left\{n,\frac{1}{\sqrt{1-x^2}}\right\}^{5/2},\label{eq:psi_deriv_estimate}\\
 |P_a'(x)|
  &\leq C(w)n\cdot\min\left\{n,\frac{1}{\sqrt{1-x^2}}\right\}^{3/2}\left(1+|a|\left((1-\sgn(a)x)\min\left\{n,\frac{1}{\sqrt{1-x^2}}\right\}+\frac{1}{n}\right)\right).\label{eq:P_a_deriv_estimate}
\end{align}
\end{propos}
\begin{proof}
We use a combination of the Bernstein (see \cite[Theorem
1.22.3]{Szego}) and A. Markov (see \cite[Theorem
5.1.8]{Borwein-Erdelyi}) inequalities, which states that for any
polynomial $p$ and every $-\alpha<x<\alpha$,
\begin{equation}\label{eq:Bernstein_Markov_inequality}
|p'(x)|\leq\deg p\cdot\min\left\{\frac{1}{\sqrt{\alpha^2 -
x^2}},\frac{\deg p}{\alpha}\right\}\max_{-\alpha\leq t\leq
\alpha}|p(t)|.
\end{equation}

Fix $-1<x<1$ and denote $\rho:=\sqrt{1-x^2}$. First we prove
\eqref{eq:phi_deriv_estimate}.  Let
$I:=\left[-\sqrt{1-\frac{1}{2}\rho^2},
\sqrt{1-\frac{1}{2}\rho^2}\,\right]$. Observe that $x\in I$, $|I|\ge
\sqrt{2}$ and $1-t^2\ge \rho^2/2$ when $t\in I$. By
\eqref{eq:phi_upper_bound},
$$\max_{t\in I}\,|\varphi(t)|\leq \max_{t\in I}\, C(w)\min\left\{n,\frac{1}{\sqrt{1-t^2}}\right\}^{1/2}\leq C(w)\min\left\{n,\frac{\sqrt{2}}{\rho}\right\}^{1/2}.$$
Thus, using \eqref{eq:Bernstein_Markov_inequality} for the interval
$I$ yields
$$|\varphi'(x)|\leq n\cdot\min\left\{\frac{\sqrt{2}}{\rho},\frac{2n}{|I|}\right\}C(w)\min\left\{n,\frac{\sqrt{2}}{\rho}\right\}^{1/2}\le C(w) n\cdot\min\left\{n,\frac{1}{\rho}\right\}^{\frac{3}{2}}.$$

The bound \eqref{eq:psi_deriv_estimate} follows similarly. To get
\eqref{eq:P_a_deriv_estimate}, note that
$$P'_a(x)=\begin{cases}\varphi'(x)-a(1-x)\psi'(x)+a\psi(x)
  &a>0\\[3pt]\varphi'(x)-a(1+x)\psi'(x)-a\psi(x)
&a<0
  \end{cases}$$
and use \eqref{eq:psi_upper_bound}, \eqref{eq:phi_deriv_estimate}
and \eqref{eq:psi_deriv_estimate}.
\end{proof}

We remark that certain lower bounds on the derivative of $P_a$ and
$\psi$ are proved in Section~\ref{sec:p_a_derivative}.

\subsection{Quadrature formulas}
In this section we will sometimes have need to consider two weight
functions simultaneously. In such cases, to avoid ambiguity, we add
the superscript $w$ to quantities such as $\xi, \eta, \lambda$ and
$\pi$ to indicate that the weight function is $w$.

\subsubsection{Distance of nodes from endpoints}
Let $S$ be the set of all non-zero polynomials $p$ satisfying
$\deg(p)\le 2n-2$ and $p\ge 0$ on $[-1,1]$. The following lemma
gives a max-min formula for $\xi_i$ involving polynomials in $S$.

\begin{lemma}\label{lem:node_max_min_formula}
For every weight function $w$,
\begin{equation}\label{eq:xi_i_max_min}
\xi_i(0)=\max_{-1<z_1<\ldots<z_{i-1}<1}\min_{\substack{p\in S\\
p(z_1)=\ldots=p(z_{i-1})=0}}\frac{\int_{-1}^1tp(t)w(t)dt}{\int_{-1}^1p(t)w(t)dt},\quad\quad
1\le i\le n.
\end{equation}
\end{lemma}

\begin{proof}
For brevity, we denote in this proof the zeros of $\varphi$ by
$\xi_i$ instead of $\xi_i(0)$.

Recall from Section~\ref{sec:background} the notation $\Sigma_0$ for
the quadrature formula of degree $2n-1$ whose nodes are $(\xi_i)$.
Using this quadrature formula to evaluate the integrals we obtain
for every $p\in S$ that
\begin{equation}\label{eq:Gaussian_quad_use}
\frac{\int_{-1}^1tp(t)w(t)dt}{\int_{-1}^1p(t)w(t)dt}=\frac{\sum_{j=1}^n\lambda(\xi_j)\xi_jp(\xi_j)}{\sum_{j=1}^n\lambda(\xi_j)p(\xi_j)}.
\end{equation}

Fix $1\le i\le n$. We prove \eqref{eq:xi_i_max_min} by establishing
that the right-hand side is both an upper and lower bound for
$\xi_i$. First, let $-1<z_1<\ldots<z_{i-1}<1$ and consider the
polynomial
\begin{equation*}
  q(t):=\prod_{j=1}^{i-1}(t-z_j)^2\cdot\prod_{j=i+1}^n(t-\xi_j)^2.
\end{equation*}
This polynomial belongs to $S$ and vanishes at $z_1,\ldots,z_{i-1}$.
Thus, using \eqref{eq:Gaussian_quad_use},
$$\min_{\substack{p\in S\\ p(z_1)=\ldots=p(z_{i-1})=0}}\frac{\int_{-1}^1tp(t)w(t)dt}{\int_{-1}^1p(t)w(t)dt}=\min_{\substack{p\in S\\ p(z_1)=\ldots=p(z_{i-1})=0}}\frac{\sum_{j=1}^n\lambda(\xi_j)\xi_jp(\xi_j)}{\sum_{j=1}^n\lambda(\xi_j)p(\xi_j)}\leq\frac{\sum_{j=1}^n\lambda(\xi_j)\xi_jq(\xi_j)}{\sum_{j=1}^n\lambda(\xi_j)q(\xi_j)}.$$
Since $q$ vanishes at $\xi_{i+1},\ldots, \xi_n$ and
$\xi_1<\cdots<\xi_i$, it follows that
$$\min_{\substack{p\in S\\ p(z_1)=\ldots=p(z_{i-1})=0}}\frac{\int_{-1}^1tp(t)w(t)dt}{\int_{-1}^1p(t)w(t)dt}\le\frac{\sum_{j=1}^i\lambda(\xi_j)\xi_j q(\xi_j)}{\sum_{j=1}^i\lambda(\xi_j)q(\xi_j)}\leq\xi_i.$$
Since the $(z_j)$ are arbitrary, we conclude that
\begin{equation}\label{eq:xi_i_formula_upper_bound}
  \xi_i\ge \max_{-1<z_1<\ldots<z_{i-1}<1}\min_{\substack{p\in S\\
p(z_1)=\ldots=p(z_{i-1})=0}}\frac{\int_{-1}^1tp(t)w(t)dt}{\int_{-1}^1p(t)w(t)dt}.
\end{equation}

Second, suppose $p$ is a polynomial in $S$ which vanishes at
$z_1=\xi_1,\ldots,z_{i-1}=\xi_{i-1}$. It follows from
\eqref{eq:Gaussian_quad_use} that
$$\frac{\int_{-1}^1tp(t)w(t)dt}{\int_{-1}^1p(t)w(t)dt}=\frac{\sum_{j=1}^n\lambda(\xi_j)\xi_jp(\xi_j)}{\sum_{j=1}^n\lambda(\xi_j)p(\xi_j)}=\frac{\sum_{j=i}^n\lambda(\xi_j)\xi_jp(\xi_j)}{\sum_{j=i}^n\lambda(\xi_j)p(\xi_j)}\geq\xi_i.$$
Hence,
\begin{equation}\label{eq:xi_i_formula_lower_bound}
\xi_i\le \min_{\substack{p\in S\\
p(\xi_1)=\ldots=p(\xi_{i-1})=0}}\frac{\int_{-1}^1tp(t)w(t)dt}{\int_{-1}^1p(t)w(t)dt}\le
\max_{-1<z_1<\ldots<z_{i-1}<1}\min_{\substack{p\in S\\
p(z_1)=\ldots=p(z_{i-1})=0}}\frac{\int_{-1}^1tp(t)w(t)dt}{\int_{-1}^1p(t)w(t)dt}.
\end{equation}
The lemma follows by putting together
\eqref{eq:xi_i_formula_upper_bound} and
\eqref{eq:xi_i_formula_lower_bound}.
\end{proof}

Denote by $u$ the constant weight function, $u\equiv 1$ on $[-1,1]$.
Let
\begin{equation}\label{eq:Legendre_poly_def}
  L_n(t):=\frac{1}{2^n
n!}\frac{d^n}{dt^n}\left[(t^2-1)^n\right]
\end{equation}
be the Legendre polynomial of degree $n$. As is well-known, the
polynomials $(L_n)$ are orthogonal with respect to $u$. By our
notation, $(\xi_i^u(0))$, $1\le i\le n$, are the roots of $L_n$. We
require the following bounds on these roots \cite[Theorem
6.21.2]{Szego},
\begin{equation}\label{eq:Legendre_nodes_bounds}
-\cos\left(\frac{2i-1}{2n+1}\pi\right)\leq\xi_i^u(0)\leq
-\cos\left(\frac{2i}{2n+1}\pi\right),\quad 1\leq i\leq n.
\end{equation}
\begin{cor}
Suppose the weight function $w$ satisfies $0<m\le w\le M$ almost
everywhere. Then for each $1\leq i\leq n$,
\begin{equation}\label{eq:xi_i_distance_endpoint_u}
\frac{m}{M}(1-\xi_i^u(0))\leq1-\xi_i^w(0)\leq\frac{M}{m}(1-\xi_i^u(0))\;\quad\text{and}\;\quad\frac{m}{M}(1+\xi_i^u(0))\leq1+\xi_i^w(0)\leq\frac{M}{m}(1+\xi_i^u(0)).
\end{equation}
Consequently, there exist absolute constants $C,c>0$ such that for
all $1\le i\le n$,
\begin{equation}\label{eq:xi_i_distance_endpoint}
\begin{split}
  c \sqrt{\frac{m}{M}}\cdot\frac{n+1-i}{n}\leq &\sqrt{1-\xi_i^w(0)}\leq
  C\sqrt{\frac{M}{m}}\cdot\frac{n+1-i}{n}\;\quad\text{and}\\
  c\sqrt{\frac{m}{M}}\cdot\frac{i}{n}\leq &\sqrt{1+\xi_i^w(0)}\leq C\sqrt{\frac{M}{m}}\cdot\frac{i}{n}.
\end{split}
\end{equation}
\end{cor}

\begin{proof} Applying Lemma~\ref{lem:node_max_min_formula} twice, once for $w$ and once for $u$, we
obtain
\begin{align*}
  1-\xi_i^w(0)&=\min_{-1<z_1<\ldots<z_{i-1}<1}\max_{\substack{p\in S\\ p(z_1)=\ldots=p(z_{i-1})=0}}\frac{\int_{-1}^1(1-t)p(t)w(t)dt}{\int_{-1}^1p(t)w(t)dt}\geq\\
  &\ge\frac{m}{M}\min_{-1<z_1<\ldots<z_{i-1}<1}\max_{\substack{p\in S\\
  p(z_1)=\ldots=p(z_{i-1})=0}}\frac{\int_{-1}^1(1-t)p(t)dt}{\int_{-1}^1p(t)dt}=\frac{m}{M}(1-\xi_i^u(0)).
\end{align*}
The proof of the other inequalities in
\eqref{eq:xi_i_distance_endpoint_u} is similar. Inequality
\eqref{eq:xi_i_distance_endpoint} now follows from
\eqref{eq:Legendre_nodes_bounds} and
\eqref{eq:xi_i_distance_endpoint_u}.
\end{proof}

\begin{propos}\label{Gauss}
Suppose the weight function $w$ satisfies $0<m\le w\le M$ almost
everywhere. Let $x:=\xi_i(a)$ for some $-\infty\le a\le \infty$ and
$1\le i\le n$. Then there exist absolute constants $C,c>0$ such that
\begin{equation*}
  \sqrt{1+x}\leq C\sqrt{\frac{M}{m}}\cdot\frac{i}{n}\quad\text{and}\quad\sqrt{1-x}\leq C\sqrt{\frac{M}{m}}\cdot\frac{n+1-i}{n}.
\end{equation*}
In addition, if $x\le \xi_n(0)$ then
\begin{equation}\label{eq:root_distance_from_1}
  \sqrt{1-x}\ge c\sqrt{\frac{m}{M}}\cdot\frac{n+1-i}{n}.
\end{equation}
Similarly, if $x\ge \xi_1(0)$ then
\begin{equation*}
  \sqrt{1+x}\ge c\sqrt{\frac{m}{M}}\cdot\frac{i}{n}.
\end{equation*}
\end{propos}

We remark that the condition $x\le \xi_n(0)$ is violated if and only
if $i=n$ and $a>0$. Indeed, one cannot expect the estimate
\eqref{eq:root_distance_from_1} to hold uniformly when $x>\xi_n(0)$
since if $i=n$ then $x\to 1$ as $a\to\infty$. A similar remark holds
for the condition $x\ge\xi_1(0)$.
\begin{proof}
  The proposition follows from \eqref{eq:xi_i_distance_endpoint}
  by using \eqref{eq:xi_a_increases} to note that if $a\ge 0$ then $\xi_i(0)\le \xi_i(a)\le
  \xi_{i+1}(0)$ and if $a\le 0$ then $\xi_{i-1}(0)\le \xi_i(a)\le
  \xi_i(0)$.
\end{proof}

\subsubsection{Weights and distances between nodes}
In this section we give upper bounds on the weights and the
inter-node distance in the quadrature formulas $\Sigma_x$. Recall
that $S_x$ is the set of nodes of $\Sigma_x$.

\begin{lemma}\label{lambda1} Suppose the weight function $w$ satisfies $0<m\le w\le M$ almost
everywhere. Then there exists an absolute constant $C>0$ such that
for every $-1<x<1$,
$$\lambda_x(-1)\leq C\frac{M^2}{m}\cdot\frac{1}{n^2}\quad\text{and}\quad\lambda_x(1)\leq C\frac{M^2}{m}\cdot\frac{1}{n^2}.$$
\end{lemma}
\begin{proof} It suffices to prove the inequality for $\lambda_x(1)$
as the inequality for $\lambda_x(-1)$ follows from it by considering
the reversed weight function $\tilde{w}(t):=w(-t)$. Let $1\le r\le
n$ and $-\infty\leq a<\infty$ be the unique numbers for which
$x=\xi_r(a)$. If $0\leq a<\infty$ then $1\notin S_x$ and hence
$\lambda_x(1)=0$ and there is nothing to prove. Suppose $-\infty\leq
a<0$. Define the following polynomial
$$q(t):=\left[\prod_{i=1}^{n-1}\left(\frac{t-\xi_i(a)}{1-\xi_i(a)}\right)^2\right]\frac{t-\xi_n(a)}{1-\xi_n(a)}.$$
Observe that $\deg q=2n-1$ and $q$ vanishes on $\xi_i(a)$ for all
$i$. Now, on the one hand, we may use the quadrature formula
$\Sigma_x$ to obtain that
\begin{equation}\label{eq:q_integral_first_time}
\int_{-1}^1 q(t)w(t)dt=\sum_{u\in S_x}\lambda_x(u)q(u)=\lambda_x(1).
\end{equation}
On the other hand, since $q\le 0$ on $[-1, \xi_n(a)]$ and $q\le 1$
on $(\xi_n(a),1]$ it follows that
\begin{equation} \label{eq:q_integral_second_time}
\int_{-1}^1 q(t)w(t)dt\leq \left(1-\xi_n(a)\right)M.
\end{equation}
The lemma follows by putting \eqref{eq:q_integral_first_time} and
\eqref{eq:q_integral_second_time} together with the fact that
$1-\xi_n(a)\leq C\frac{M}{m}\cdot\frac{1}{n^2}$ by Proposition
\ref{Gauss}.
\end{proof}

\begin{lemma}\label{lambda} Suppose the weight function $w$ satisfies $w\le M$ almost
everywhere. Then there exists an absolute constant $C>0$ such that
\begin{equation}\label{eq:Christoffel_function_bounds}
\lambda(x)\leq \frac{C
M}{n}\max\left\{\sqrt{1-x^2},\frac{1}{n}\right\},\quad -1<x<1.
\end{equation}
\end{lemma}
We remark that this bound is not sharp solely under the condition
that $w\le M$ almost everywhere. For instance, for the Jacobi weight
$w(x)=(1-x)^{\alpha}(1+x)^{\beta}$ with $0<\alpha,\beta\le1/2$ we have that
$\lambda(x)$ is of order $n^{-(2+2\max\{\alpha,\beta\})}$ near one
of the endpoints of the interval \cite[(15.3.1), (4.21.7) and
Theorem 8.21.13]{Szego}. However, the bound is sharp up to the value
of the constant if one imposes some additional assumptions on $w$
(in particular, in the cases of interest in our main theorems), see
Corollary \ref{cor:lambda_lower}.

\begin{proof} As in the previous section, we denote by
$u$ the constant weight function, $u\equiv 1$ on $[-1,1]$. Fix
$-1<x<1$. Let $S$ be the set of all polynomials $f$ of degree $\leq
2n-1$ that are non-negative on $[-1,1]$ and satisfy $f(x)=1$. It is
well known that $\lambda^w(x)=\min_{f\in S}\int_{-1}^1f(t)w(t)dt$
(see \cite[Chapter II, section 4]{Karlin-Studden}). Therefore
$$\lambda^w(x)=\min_{f\in S}\int_{-1}^1f(t)w(t)dt\leq M\min_{f\in S}\int_{-1}^1f(t)dt=M\lambda^u(x).$$
Thus it suffices to prove \eqref{eq:Christoffel_function_bounds} for
the weight function $u$. To this end, let $1\le i\le n$ and
$-\infty\leq a<\infty$ be the unique numbers for which
$x=\xi^u_i(a)$. We may assume without loss of generality that $n\ge
3$ since otherwise the lemma is trivial. We consider separately
three cases, in all of which we rely on the
Chebyshev-Markov-Stieltjes inequalities
\eqref{eq:Chebyshev_Markov_Stieltjes_intro} and
Proposition~\ref{Gauss}.
\begin{enumerate}
  \item If $i=1,2$,
\begin{equation*}
  \lambda^u(x)\leq\underline{\pi}^u(\xi^u_{3}(a))\leq\int_{-1}^{\xi^u_{3}(a)}dt=\xi^u_{3}(a)+1\leq
  \frac{C}{n^2}.
\end{equation*}
  \item If $i=n-1,n$,
  \begin{equation*}
    \lambda^u(x)\leq 2-\pi^u(\xi^u_{n-2}(a))\leq
    2-\int_{-1}^{\xi^u_{n-2}(a)}dt=1-\xi^u_{n-2}(a)\leq\frac{C}{n^2}.
  \end{equation*}
  \item If $2<i<n-1$, by \eqref{eq:Legendre_nodes_bounds} we have
  \begin{align*}
    \lambda^u(x)&=\underline{\pi}^u(\xi^u_{i+1}(a))-\pi^u(\xi^u_{i-1}(a))\leq\int_{-1}^{\xi^u_{i+1}(a)} dt-\int_{-1}^{\xi^u_{i-1}(a)} dt=\xi^u_{i+1}(a)-\xi^u_{i-1}(a)\le\\
    &\leq\xi^u_{i+2}(0)-\xi^u_{i-2}(0)\leq
    \cos\left(\frac{2(i-2)-1}{2n+1}\pi\right)-\cos\left(\frac{2(i+2)}{2n+1}\pi\right)\le\\
    &\le C\frac{i(n+1-i)}{n^3}\le
    \frac{C}{n}\sqrt{1-x^2}.\qedhere
  \end{align*}
\end{enumerate}
\end{proof}

The next lemma deduces upper bounds on the inter-node distance in
the quadrature formulas $\Sigma_x$. In the case $a=0$ (from which
the general case follows using \eqref{eq:xi_a_increases}), these
bounds appear in the work of Erd\H os and Tur\'an
\cite{Erdos-Turan}.
\begin{lemma}\label{"erdos-turan"}
Suppose the weight function $w$ satisfies $0<m\le w\le M$ almost
everywhere. Then there exists an absolute constant $C>0$ such that
for every $1\leq i\leq n-1$ and every $-\infty\leq a\leq \infty$:
\begin{equation}\label{eq:upper_bound_on_separation}
  \xi_{i+1}(a)-\xi_i(a)\leq
C\left(\frac{M}{m}\right)^2\frac{i(n-i)}{n^3}.
\end{equation}
\end{lemma}
\begin{proof}
Assume first that $\xi_i(a),\xi_{i+1}(a)\in(-1,1)$. The
Chebyshev-Markov-Stieltjes inequalities
\eqref{eq:Chebyshev_Markov_Stieltjes_intro} imply that
$$m\left[\xi_{i+1}(a)-\xi_i(a)\right]\leq\int_{\xi_i(a)}^{\xi_{i+1}(a)}w(t)dt=\int_{-1}^{\xi_{i+1}(a)}w(t)dt-\int_{-1}^{\xi_i(a)}w(t)dt\leq$$
$$\leq\pi\left(\xi_{i+1}(a)\right)-\underline{\pi}\left(\xi_i(a)\right)=\lambda\left(\xi_i(a)\right)+\lambda\left(\xi_{i+1}(a)\right)$$
and the result follows from Lemma \ref{lambda} and Proposition
\ref{Gauss}. Second, if either $\xi_i(a)$ or $\xi_{i+1}(a)$ are in
$\{-1,1\}$ the result follows from \eqref{eq:xi_a_increases}.
\end{proof}
We do not know if the bound \eqref{eq:upper_bound_on_separation} is
sharp under the conditions of Lemma~\ref{"erdos-turan"}, however,
the bound is sharp up to a constant depending only on $w$ if $w$
satisfies some additional assumptions, e.g., the conditions of
Theorem~\ref{thm:Badkov_thm} with $\alpha=0$. This may be deduced in
two different ways from arguments in this paper. First, it follows
from Proposition~\ref{propos:root_sep} below, as follows. If $a\ge 0$ we may take $b_+=0$
and the limit $b_-\to-\infty$ in \eqref{eq:separation_negative} and
use \eqref{eq:xi_a_increases}. If $a\le 0$ we may similarly take
$a_- = 0$ and the limit $a_+\to\infty$ in
\eqref{eq:separation_positive} and use \eqref{eq:xi_a_increases}.
Second, it may be deduced from Lemma \ref{segment}.

\section{Separation of nodes of quadrature formulas}\label{sec:root_separation}
Part of the motivation for this paper came out of the works
\cite{Kuijlaars1} and \cite{Kuijlaars2} of Kuijlaars. There, a lower
bound for $\pi'$ is established for the special case of a Jacobi
weight function $w(x)=(1-x)^{\alpha}(1+x)^{\beta}$ with
$\alpha,\beta\ge0$ (\cite[Proposition 7.2]{Kuijlaars1} for the
ultraspherical case $\alpha=\beta$, \cite[Lemma 5.1]{Kuijlaars2} for
the general case). Another ingredient appearing in those works are
results (\cite[Proposition 7.3]{Kuijlaars1}, \cite[Lemma
5.2]{Kuijlaars2}) bounding the distance between nodes of different
canonical representations. In this short section, which is not used
in the proofs of our main theorems, we observe that such a result
holds also for the more general weight functions which we consider.

\begin{propos}\label{propos:root_sep} Suppose $w$ satisfies the conditions of
Theorem~\ref{thm:Badkov_thm} with $\alpha=0$. Then there exists a
constant $c(w)>0$ such that the following holds.
\begin{enumerate}
\item For every $1\leq
i\leq n$ and $0\le a_-<a_+<\infty$,
\begin{equation}\label{eq:separation_positive}
\xi_i(a_+)-\xi_i(a_-)\geq
c(w)\frac{(n+1-i)^2(a_+-a_-)}{n^3\left(1+\frac{n+1-i}{i}a_-\right)\left(1+\frac{n+1-i}{i}a_+\right)}.
\end{equation}
\item For every $1\leq
i\leq n$ and $-\infty<b_-<b_+\le 0$,
\begin{equation}\label{eq:separation_negative}
\xi_i(b_+)-\xi_i(b_-)\geq
c(w)\frac{i^2(b_+-b_-)}{n^3\left(1+\frac{i}{n+1-i}(-b_-)\right)\left(1+\frac{i}{n+1-i}(-b_+)\right)}.
\end{equation}
\end{enumerate}
\end{propos}
\begin{proof} We only prove \eqref{eq:separation_positive}. The proof of
\eqref{eq:separation_negative} is similar.

Fix $1\le i\le n$ and $0\le a_-<a_+<\infty$. The proof proceeds by
bounding $P_{a_-}(\xi_i(a_+))-P_{a_-}(\xi_i(a_-))$ from below and
from above (alternatively one can bound the same expression with
$P_{a_-}$ replaced by $P_{a_+})$.

On the one hand, by \eqref{eq:xi_a_increases} and Lagrange's mean
value theorem,
$$P_{a_-}(\xi_i(a_+))-P_{a_-}(\xi_i(a_-))=P_{a_-}'(t)(\xi_i(a_+)-\xi_i(a_-))$$
for some $\xi_i(a_-)<t<\xi_i(a_+)$. Hence by Proposition
\ref{Bernstein} and Proposition \ref{Gauss},
\begin{align}\label{eq:combination_of_P_a_upper_bound}
  &\lvert P_{a_-}(\xi_i(a_+))-P_{a_-}(\xi_i(a_-))\rvert=\lvert
  P_{a_-}'(t)\rvert(\xi_i(a_+)-\xi_i(a_-))\leq\nonumber\\
  &\leq C(w)
n\cdot\min\left\{n,\frac{1}{\sqrt{1-t^2}}\right\}^{\frac{3}{2}}\left(1+a_-\left((1-t)\min\left\{n,\frac{1}{\sqrt{1-t^2}}\right\}+\frac{1}{n}\right)\right)(\xi_i(a_+)-\xi_i(a_-))\leq\nonumber\\
  &\leq C(w)
n\left(\frac{n^2}{i(n+1-i)}\right)^{\frac{3}{2}}\left(1+a_-\frac{n+1-i}{i}\right)(\xi_i(a_+)-\xi_i(a_-)).
\end{align}

On the other hand, using the fact that
$P_{a_-}(\xi_i(a_-))=P_{a_+}(\xi_i(a_+))=0$ by definition and
substituting the definition \eqref{eq:P_a_def} of $P_a$ gives
\begin{equation}\label{eq:P_a_expression_psi}
P_{a_-}(\xi_i(a_+))-P_{a_-}(\xi_i(a_-))=P_{a_-}(\xi_i(a_+))-P_{a_+}(\xi_i(a_+))=(a_+-a_-)(1-\xi_i(a_+))\psi(\xi_i(a_+)).
\end{equation}
Similarly, we may obtain an expression in terms of $\varphi$ by
writing,
\begin{equation}\label{eq:P_a_expression_phi}
P_{a_-}(\xi_i(a_+))-P_{a_-}(\xi_i(a_-))=P_{a_-}(\xi_i(a_+))-\frac{a_-}{a_+}P_{a_+}(\xi_i(a_+))=\frac{a_+-a_-}{a_+}\varphi(\xi_i(a_+)).
\end{equation}
Combining \eqref{eq:P_a_expression_psi} and
\eqref{eq:P_a_expression_phi}, applying Corollary~\ref{cor:Badkov}
and Proposition~\ref{Gauss} yields
\begin{align*}
&\left(a_{+}+(1+\xi_i(a_+))\min\left\{n,\frac{1}{\sqrt{1-\xi_i(a_+)^2}}\right\}\right)\lvert
P_{a_-}(\xi_i(a_+))-P_{a_-}(\xi_i(a_-))\rvert=\\
&=(a_+-a_-)\left[|\varphi(\xi_i(a_+))|+\left(1-\xi_i(a_+)^2\right)\min\left\{n,\frac{1}{\sqrt{1-\xi_i(a_+)^2}}\right\}|\psi(\xi_i(a_+))|\right]\ge\\
&\ge c(w)(a_+ -
a_-)\min\left\{n,\frac{1}{\sqrt{1-\xi_i(a_+)^2}}\right\}^{1/2}\ge
c(w)(a_+ - a_-)\frac{n}{\sqrt{i(n+1-i)}}.
\end{align*}
In addition, by Proposition~\ref{Gauss},
\begin{equation*}
  (1+\xi_i(a_+))\min\left\{n,\frac{1}{\sqrt{1-\xi_i(a_+)^2}}\right\}\le
  C(w)\frac{i}{n+1-i}.
\end{equation*}
Putting together the last two inequalities we finally arrive at
\begin{equation}\label{eq:combination_of_P_a_lower_bound}
  \lvert
P_{a_-}(\xi_i(a_+))-P_{a_-}(\xi_i(a_-))\rvert\ge c(w)\frac{(a_+ -
a_-)\frac{n}{\sqrt{i(n+1-i)}}}{a_{+} + \frac{i}{n+1-i}}.
\end{equation}

Comparing \eqref{eq:combination_of_P_a_upper_bound} and
\eqref{eq:combination_of_P_a_lower_bound} shows that
$$C(w)\cdot n\cdot\left(\frac{n^2}{i(n+1-i)}\right)^{\frac{3}{2}}\left(1+a_-\frac{n+1-i}{i}\right)(\xi_i(a_+)-\xi_i(a_-))\geq c(w)\frac{(a_+-a_-)\frac{n}{\sqrt{i(n+1-i)}}}{a_++\frac{i}{n+1-i}},$$
from which \eqref{eq:separation_positive} follows.
\end{proof}

\section{Estimating $P_a'$ and the interpolation polynomials}\label{sec:p_a_derivative}

In this section we prove a lower bound on $|P_a'|$ at roots of
$P_a$. This lower bound will be used in the next section to prove
Theorem~\ref{thm:Lipschitz} and Theorem~\ref{thm:abs_cont}. We also give bounds for certain
interpolation polynomials defined below.

We recall the definition of the sign function from
\eqref{eq:sgn_a_def}. Define also, for real $x$, the truncation
operation,
\begin{equation}\label{eq:x_bar_def}
  \overline{x} := \begin{cases}
    \xi_1(0)&x\le \xi_1(0)\\
    x&x\in[\xi_1(0),\xi_n(0)]\\
    \xi_n(0)&x\ge \xi_n(0)
  \end{cases}.
\end{equation}
\begin{lemma}\label{P_a}
Suppose $w$ satisfies the conditions of Theorem~\ref{thm:Badkov_thm}
with $\alpha=0$. Then there exists a constant $c(w)>0$ such that the
following holds for every $-1<x<1$.
\begin{enumerate}
\item If $x=\xi_r(a)$ for $1\leq r\leq n$ and
$-\infty<a<\infty$ then
\begin{equation*}
  |P'_a(x)|\ge c(w)\sqrt{\frac{n}{\lambda(x)}}\max\left\{\frac{|a|}{1+\sgn(a)x},
    \frac{1}{\sqrt{1-\overline{x}^2}}\right\}.
\end{equation*}
\item If $x=\eta_r$ for $1\leq r\leq n-1$ then
$$|\psi'(x)|\geq c(w)\sqrt{\frac{n}{\lambda(x)}}\cdot\frac{1}{1-x^2}.$$
\end{enumerate}
\end{lemma}
Our proof of the lemma relies on lower bounds for $|\varphi|$ and
$|\psi|$. The bound \eqref{eq:phi_psi_lower_bound_new} shows that
$|\varphi|$ and $|\psi|$ cannot be simultaneously small. Thus, near
a root of one, the other must be large. The next lemma makes this
idea precise.
\begin{lemma}\label{segment}
Suppose $w$ satisfies the conditions of Theorem~\ref{thm:Badkov_thm}
with $\alpha=0$. Then there exists a constant $c(w)>0$ such that the
following holds.
\begin{enumerate}
\item For every $1\leq r\leq n-1$ the interval $I:=\big\{x\,:\,\lvert x-\eta_r\rvert\leq c(w)\frac{r(n+1-r)}{n^3}\big\}$ satisfies $I\subseteq [-1,1]$ and
$$\min_{x\in I}\, |\varphi(x)|\geq c(w)\frac{n}{\sqrt{r(n+1-r)}}.$$
\item For every $1\leq r\leq n$ the interval $I:=\big\{x\,:\,\lvert x-\xi_r(0)\rvert\leq c(w)\frac{r(n+1-r)}{n^3}\big\}$ satisfies $I\subseteq [-1,1]$ and
$$\min_{x\in I}\,\lvert \psi(x)\rvert\geq c(w)\left(\frac{n}{\sqrt{r(n+1-r)}}\right)^3.$$
\end{enumerate}
\end{lemma}

\begin{proof}
We prove only the first part. The proof of the second part is
similar.

Denote $\rho:=\frac{r(n+1-r)}{n^2}$. Let $\eps = \eps(w)>0$ be a
constant, depending on $w$ but independent of $n$, whose value is
sufficiently small for the following calculations. Define
\begin{equation*}
  I_1 := \left\{x\,:\, |x-\eta_r|\le \eps\frac{\rho}{n}\right\}
  \quad\text{and}\quad I_2:=\left\{x\,:\, 1-x^2\ge \frac{1}{2}(1-\eta_r^2)\right\}.
\end{equation*}
Observe that $c(w)\rho^2\le 1-\eta_r^2\le C(w)\rho^2$ by
\eqref{eq:xi_a_increases} and Proposition~\ref{Gauss}. Thus, since
$\psi(\eta_r)=0$, Corollary~\ref{cor:Badkov} implies that
\begin{equation*}
  \lvert \varphi(\eta_r)\rvert\geq c(w)\min\left\{n,\frac{1}{\sqrt{1-{\eta_r}^2}}\right\}^{1/2}\geq\frac{c(w)}{\sqrt{\rho}}.
\end{equation*}
In addition, Proposition~\ref{Bernstein} implies that
\begin{equation*}
  |\varphi'(x)|\leq C(w)n\cdot\min\left\{n,\frac{1}{\sqrt{1-x^2}}\right\}^{\frac{3}{2}}\leq
  C(w)\frac{n}{\rho^{3/2}},\quad x\in I_2.
\end{equation*}
Finally, since $\rho\ge \frac{1}{n}$, it follows that $I_1\subseteq
I_2$ if $\eps$ is sufficiently small. Thus, for sufficiently small
$\eps$,
\begin{equation*}
  |\varphi(x)|\ge |\varphi(\eta_r)| - |x-\eta_r|\max_{y\in I_2} |\varphi'(y)|
  \ge \frac{c(w)}{\sqrt{\rho}} - \eps \frac{\rho}{n}
  C(w)\frac{n}{\rho^{3/2}} \ge \frac{c(w)}{\sqrt{\rho}},\quad x\in
  I_1.\qedhere
\end{equation*}
\end{proof}

A second tool in our proof of Lemma~\ref{P_a} is the following
polynomial, which is a relative of the Lagrange interpolation
polynomial for $P_a$.

Recall that $S_x$ is the set of nodes of $\Sigma_x$ and recall the
definition of the index function $I$ from \eqref{eq:index_fcn_def}.
For $-1<x<1$ define the polynomial
\begin{equation}\label{eq:q_x_def}
q_x(t):=\prod_{u\in S_x-\{x\}}\left(\frac{t-u}{x-u}\right)^{I(u)}.
\end{equation}

\begin{figure}[t]
\centering
{\includegraphics[width=0.7\textwidth]{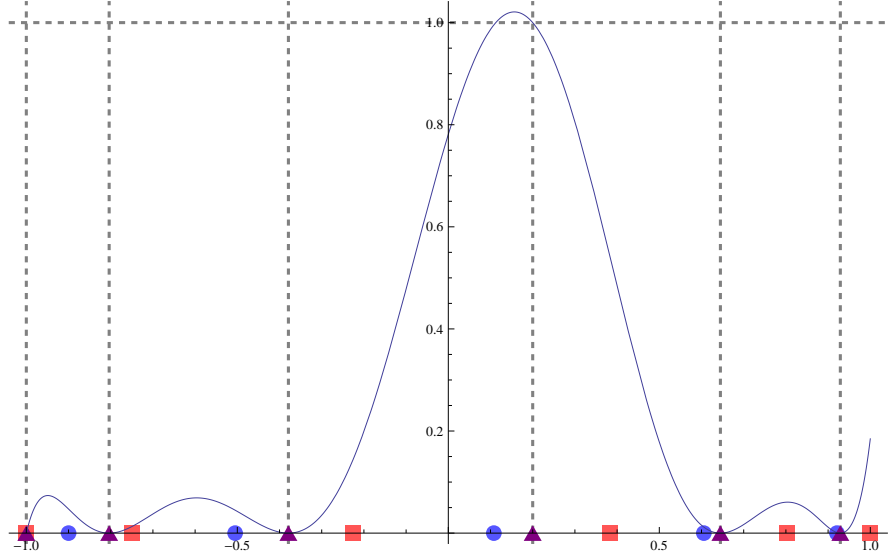}}
{\caption{A plot of the polynomial $q_x$ for $n=5$, $x=0.2$ and the
weight function $w(t)=\max\{1,1+4t\}$. The circles on the axis
denote the nodes of the Gaussian quadrature, the squares denote the
nodes of the Lobatto quadrature and the triangles denote the nodes
of the canonical representation $\Sigma_{0.2}$. Observe that, by
definition, $q_x(x)=1$ and $q_x'$ is zero at all nodes of
$\Sigma_{0.2}-\{-1,x,1\}$. \label{fig:q_x}} }
\end{figure}

Figure~\ref{fig:q_x} shows a plot of $q_x$ for a certain choice of
the parameters. We list some straightforward properties of $q_x$.
\begin{enumerate}
  \item $\deg(q_x) = \sum_{u\in S_x} I(x) - 2\le 2n-1$.
  \item $q_x(x)=1$, $q_x(u)=0$ for
$u\in S_x-\{x\}$ and $q'_x(u)=0$ for $u\in (S_x-\{x\})\cap(-1,1)$.
Furthermore, $q_x$ is the (unique) least degree polynomial
satisfying these equalities.
  \item $q_x\ge 0$ on $[-1,1]$.
  \item Writing $x=\xi_r(a)$ we have the following formula
\begin{equation}\label{eq:Lagrange_representation}
\begin{split}
  q_x(t)
&=\begin{cases}
    \frac{1+\sgn(a)t}{1+\sgn(a)x}\left(\frac{P_a(t)}{(t-x)P'_a(x)}\right)^2 & -\infty<a<\infty\\
    \frac{1-t^2}{1-x^2}\left(\frac{\psi(t)}{(t-x)\psi'(x)}\right)^2&|a|=\infty
  \end{cases}.
  \end{split}
\end{equation}
\end{enumerate}

Observe that by the first two properties above, using the quadrature
formula $\Sigma_x$, we have
\begin{equation}\label{eq:q_x_integral}
\int_{-1}^1q_x(t)w(t)dt=\sum_{u\in
S_x}\lambda_x(u)q_x(u)=\lambda(x).
\end{equation}

\begin{lemma}\label{lem:P'a_via_Pa}
Suppose the weight function $w$ satisfies $0<m\le w\le M$ almost
everywhere. Let $1\leq r\leq n$ and let $2\leq r'\leq n-1$ satisfy
$|r'-r|\leq 1$. Suppose that $I$ is an interval satisfying
$I\subseteq[\xi_{r'-1}(0),\xi_{r'+1}(0)]$. Then,
\begin{enumerate}
\item For every $-\infty<a<\infty$,
\begin{equation}\label{eq:P'a_via_Pa}|P'_a(\xi_r(a))|\ge c(w)\sqrt{|I|}\frac{n^3}{r(n+1-r)}\frac{1}{\sqrt{\lambda(\xi_r(a))}}\min_{t\in I}|P_{a}(t)|.\end{equation}
\item If $1\leq r\leq n-1$,
\begin{equation*}
|\psi'(\eta_r)|\ge c(w)\sqrt{|I|}\frac{n^3}{r(n+1-r)}\frac{1}{\sqrt{\lambda(\eta_r)}}\min_{t\in I}|\psi(t)|.
\end{equation*}
\end{enumerate}
\end{lemma}
The inequalities in the lemma may be trivial, in the sense that
their right-hand sides may vanish, but we will avoid this
possibility in our usage by choosing $I$ appropriately.
\begin{proof}
We prove only the first part of the lemma. The second part is
similar. Observe that by \eqref{eq:q_x_integral}, the non-negativity
of $q_{\xi_r(a)}$ and our assumption that $w$ is bounded below, we
have
\begin{equation}\label{eq:q_x_upper_bound}
  \lambda(\xi_r(a)) \ge \int_I q_{\xi_r(a)}(t)w(t)dt \ge |I|\min_{t\in I} \left(q_{\xi_r(a)}(t)w(t)\right) \ge c(w)|I|\min_{t\in I} q_{\xi_r(a)}(t).
\end{equation}
We continue to estimate the minimum. By \eqref{eq:Lagrange_representation},
\begin{equation}\label{eq:q_x_lower_bound}
  q_{\xi_r(a)}(t) = \frac{1+\sgn(a)t}{1+\sgn(a)\xi_r(a)}\cdot\frac{1}{(t-\xi_r(a))^2}\cdot P_{a}(t)^2\cdot \frac{1}{P'_a(\xi_r(a))^2}.
\end{equation}
Let us estimate each of the first two factors above separately.
First, observe that by the conditions of the lemma and
Proposition~\ref{Gauss},
\begin{equation}\label{eq:q_x_first}
  \min_{t\in I} \frac{1+\sgn(a)t}{1+\sgn(a)\xi_r(a)} \ge c(w).
\end{equation}
Second, note that by the conditions of the lemma and
Lemma~\ref{"erdos-turan"},
\begin{equation}\label{eq:q_x_second}
  \max_{t\in I} |t-\xi_r(a)| \le  \max\{|\xi_{r'+1}(0) - \xi_r(a)|,|\xi_{r'-1}(0) - \xi_r(a)|\}\le C(w) \frac{r(n+1-r)}{n^3}.
\end{equation}
Substituting \eqref{eq:q_x_first} and \eqref{eq:q_x_second} into \eqref{eq:q_x_lower_bound} we obtain
\begin{equation*}
  \min_{t\in I} q_{\xi_r(a)}(t) \ge c(w)\left(\frac{n^3}{r(n+1-r)}\right)^2
  \frac{1}{P'_a(\xi_r(a))^2}\left(\min_{t\in I}|P_a(t)|\right)^2.
\end{equation*}
Finally, we may continue \eqref{eq:q_x_upper_bound} to obtain
\begin{equation*}
  \lambda(\xi_r(a)) \ge c(w)|I|\left(\frac{n^3}{r(n+1-r)}\right)^2
  \cdot\frac{1}{P'_a(\xi_r(a))^2}\left(\min_{t\in I}|P_a(t)|\right)^2
\end{equation*}
and the lemma follows.
\end{proof}

\begin{proof}[Proof of Lemma~\ref{P_a}]

We prove only the first part of the lemma. The second part is
similar and even simpler.
To simplify the presentation, take $2\leq r'\leq n-1$ such that
$|r'-r|\leq 1$ (when $n=1,2$ the definitions of $\tilde{I}_\varphi$
and $\tilde{I}_\psi$ below may be adjusted properly and the proof of
Lemma~\ref{lem:P'a_via_Pa} repeated for them. We omit the details).
By Lemma \ref{segment} there is a constant $c(w)>0$ such that
\begin{align}
\min_{x\in I_{\varphi}}\, |\varphi(x)|&\geq c(w)\frac{n}{\sqrt{r'(n+1-r')}}\label{eq:I_phi_inequality},\\
\min_{x\in I_{\psi}}\,\lvert \psi(x)\rvert&\geq
c(w)\left(\frac{n}{\sqrt{r'(n+1-r')}}\right)^3.\label{eq:I_psi_inequality}
\end{align}
where
\begin{align*}
I_{\varphi}&:=\left\{x\,:\,\lvert x-\eta_{r'}\rvert\leq c(w)\frac{r'(n+1-r')}{n^3}\right\},\\
I_{\psi}&:=\left\{x\,:\,\lvert x-\xi_{r'}(0)\rvert\leq
c(w)\frac{r'(n+1-r')}{n^3}\right\}.
\end{align*}
In particular, $\varphi$ does not change sign in $I_{\varphi}$ and
$\psi$ does not change sign in $I_{\psi}$. Consequently, since
$\eta_{r'}\in I_{\varphi}$ and $\xi_{r'}(0)\in I_{\psi}$, then
necessarily
\begin{equation}\label{eq:subseteq1}
I_{\varphi}\subseteq[\xi_{r'}(0),\xi_{r'+1}(0)]\quad\text{and}\quad
I_{\psi}\subseteq[\eta_{r'-1},\eta_{r'}].
\end{equation}
Thus $\psi$ changes sign in $I_{\varphi}$ exactly once, at the
midpoint $\eta_{r'}$, and $\varphi$ changes sign in $I_{\psi}$
exactly once, at the midpoint $\xi_{r'}(0)$.

Let $\tilde{I}_{\varphi}$ be the sub-segment of $I_{\varphi}$ in
which $\varphi$ and $a\,\psi$ have opposite signs. Precisely,
\begin{equation*}
  \tilde{I}_{\varphi}:=\{x\in I_{\varphi}\,:\,a\varphi(x)\psi(x)\le 0\}.
\end{equation*}
(the sub-segment to the left of $\eta_{r'}$ if $a<0$ or the
sub-segment to the right of $\eta_{r'}$ if $a>0$ or the whole
$I_\varphi$ if $a=0$, see Figure~\ref{fig:phi_psi}). In the same
manner, let $\tilde{I}_{\psi}$ be the sub-segment of $I_{\psi}$ in
which $\varphi$ and $a\,\psi$ have opposite signs. Then, using
\eqref{eq:P_a_def} and \eqref{eq:I_phi_inequality} and the fact that
$|r'-r|\le 1$,
\begin{equation}
  \min_{t\in \tilde{I}_\varphi} |P_a(t)| \ge \min_{t\in \tilde{I}_\varphi}\lvert \varphi(t)\rvert \ge c(w)\frac{n}{\sqrt{r(n+1-r)}}\label{eq:PaIvarphi}\end{equation}
and, similarly, using \eqref{eq:P_a_def},
\eqref{eq:I_psi_inequality} and \eqref{eq:subseteq1},
\begin{align}
\min_{t\in \tilde{I}_\psi} |P_a(t)| &\ge |a|\min_{t\in \tilde{I}_\psi}(1-\sgn(a)t)\lvert \psi(t)\rvert\geq\nonumber\\
&\ge
c(w)|a|(1-\max\{\sgn(a)\eta_{r'-1},\sgn(a)\eta_{r'}\})\left(\frac{n}{\sqrt{r(n+1-r)}}\right)^3.\label{eq:PaIpsi}
\end{align}
Observe that by our definitions,
\begin{equation*}
   \min(|\tilde{I}_\varphi|, |\tilde{I}_\psi|) \ge
   c(w)\frac{r(n+1-r)}{n^3}.
\end{equation*}
Note also that $\tilde{I}_{\varphi}$ and $\tilde{I}_{\psi}$ satisfy
the assumptions of Lemma \ref{lem:P'a_via_Pa} by
\eqref{eq:subseteq1}. Thus, by plugging \eqref{eq:PaIvarphi} and
\eqref{eq:PaIpsi}, respectively, in \eqref{eq:P'a_via_Pa}, and using
Proposition \ref{Gauss}, we have that
\begin{align}
 |P'_a(\xi_r(a))|&\ge c(w)\frac{n}{\sqrt{r(n+1-r)}}\sqrt{\frac{n}{\lambda(\xi_r(a))}}\min_{t\in \tilde{I}_\varphi} |P_a(t)|\ge \nonumber\\
&\geq
c(w)\frac{n^2}{r(n+1-r)}\cdot\sqrt{\frac{n}{\lambda(\xi_r(a))}}\geq
c(w)\sqrt{\frac{n}{\lambda(\xi_r(a))}}\cdot\frac{1}{\sqrt{1-\overline{\xi_r(a)}^2}}
\label{eq:right_of_max}
\end{align}
and
\begin{align}
|P'_a(\xi_r(a))|&\ge c(w)\frac{n}{\sqrt{r(n+1-r)}}\sqrt{\frac{n}{\lambda(\xi_r(a))}}\min_{t\in \tilde{I}_\psi} |P_a(t)| \ge\nonumber\\
&\ge c(w)\left(\frac{n^2}{r(n+1-r)}\right)^2\sqrt{\frac{n}{\lambda(\xi_r(a))}}|a|(1-\max\{\sgn(a)\eta_{r'-1},\sgn(a)\eta_{r'}\})\geq\nonumber\\
&\geq
c(w)\sqrt{\frac{n}{\lambda(\xi_r(a))}}\cdot\frac{|a|}{1+\sgn(a)\xi_r(a)}.
\label{eq:left_of_max}
\end{align}
The lemma follows by putting together \eqref{eq:left_of_max} and
\eqref{eq:right_of_max}.
\end{proof}

It is useful to note that combining Lemma~\ref{P_a} with
Proposition~\ref{Bernstein} we may obtain a lower bound on the
function $\lambda(x)$, matching the bound given by
Lemma~\ref{lambda} up to a constant depending on $w$. This is
embodied in the following corollary which is probably well-known to
experts in the field.
\begin{cor}\label{cor:lambda_lower} Suppose $w$ satisfies the conditions of Theorem~\ref{thm:Badkov_thm}
with $\alpha=0$. Then there exists a constant $c(w)>0$ such that
\begin{equation}\label{eq:lambda_lower_bound}
\lambda(x)\geq
\frac{c(w)}{n}\max\left\{\sqrt{1-x^2},\frac{1}{n}\right\},\quad
-1<x<1.
\end{equation}
\end{cor}
\begin{proof} Suppose $x=\xi_r(a)$ for some $1\leq r\leq n$ and $-\infty<a<\infty$. By Lemma \ref{P_a},
\begin{equation}\label{eq:P_prime_a_lower_bound_in_cor}
  |P'_a(x)|\ge c(w)\sqrt{\frac{n}{\lambda(x)}}\max\left\{\frac{|a|}{1+\sgn(a)x},
    \frac{1}{\sqrt{1-\overline{x}^2}}\right\}.
\end{equation}
Observe that by Proposition~\ref{Gauss},
$\min\left\{n,\frac{1}{\sqrt{1-x^2}}\right\}\leq
\frac{C(w)}{\sqrt{1-\overline{x}^2}}$. Thus, by Proposition
\ref{Bernstein},
\begin{align*}|P_a'(x)|&\leq C(w)n\cdot\min\left\{n,\frac{1}{\sqrt{1-x^2}}\right\}^{3/2}\left(1+|a|\left((1-\sgn(a)x)\min\left\{n,\frac{1}{\sqrt{1-x^2}}\right\}+\frac{1}{n}\right)\right)\leq\\
&\leq C(w)n\cdot\min\left\{n,\frac{1}{\sqrt{1-x^2}}\right\}^{1/2}\left(\frac{1}{\sqrt{1-\overline{x}^2}}+|a|\left(\frac{1-\sgn(a)x}{1-x^2}+1\right)\right)\leq\\
&\leq
C(w)n\cdot\min\left\{n,\frac{1}{\sqrt{1-x^2}}\right\}^{1/2}\max\left\{\frac{1}{\sqrt{1-\overline{x}^2}},\frac{|a|}{1+\sgn(a)x}\right\}.
\end{align*}
Combined with \eqref{eq:P_prime_a_lower_bound_in_cor} this implies
\eqref{eq:lambda_lower_bound}. Now suppose $x=\eta_r$ for some
$1\leq r\leq n-1$. Then the result follows since by Lemma \ref{P_a},
$$|\psi'(x)|\geq c(w)\sqrt{\frac{n}{\lambda(x)}}\cdot\frac{1}{1-x^2},$$
and by Proposition \ref{Bernstein},
\begin{equation*}
|\psi'(x)|\leq C(w)
  n\cdot\min\left\{n,\frac{1}{\sqrt{1-x^2}}\right\}^{5/2}\leq  C(w)
 \frac{n}{1-x^2}\cdot\min\left\{n,\frac{1}{\sqrt{1-x^2}}\right\}^{1/2}.\qedhere
\end{equation*}
\end{proof}

\subsection{Estimating the interpolation polynomials}

In this section we explore the localization properties of the
interpolation polynomials $q_x$. Specifically, we show in the next
proposition that $q_x$ is everywhere bounded by a constant and is in
fact much smaller away from the point $x$. The proof of
Theorem~\ref{thm:Lipschitz} does not require the results of this
section but these results are used in the proofs of
Theorem~\ref{thm:abs_cont} and Theorem~\ref{thm:discont}.

\begin{propos}\label{propos:q bound for t far from x} Suppose $w$ satisfies the conditions of Theorem~\ref{thm:Badkov_thm}
with $\alpha=0$. Then there exists a constant $C(w)>0$ such that for
any $-1<t,x<1$, $x\neq t$, the following bounds hold
\begin{equation}\label{eq:q_bound_far2}q_x(t)\leq C(w),\end{equation}
and
\begin{equation}\label{eq:q_bound_far1}q_x(t)\leq\frac{C(w)}{n\max\left\{1,n\sqrt{1-t^2}\right\}(t-x)^2}.\end{equation}
\end{propos}

We start with the following lemma which is inspired by Erd\H os and
Lengyel \cite{Erdos-Lengyel}.

\begin{lemma}\label{lem:q naive bound} Suppose the weight function $w$ satisfies $0<m\le w\le M$ almost
everywhere. Then for every $-1<t,x<1$,
$$q_x(t)\leq\frac{M}{m}\cdot\frac{1+\sgn(x-t)x}{1+\sgn(x-t)t}.$$
\end{lemma}

\begin{proof}
We show the proof for $t\leq x$. The proof for $x\leq t$ is similar.
Let $g(u):=q_x\left(-1+\frac{1+t}{1+x}(1+u)\right)$. Since $\deg g=\deg q_x\leq 2n-1$ and $g, q_x\geq 0$ on $[-1,1]$ we get, using the quadrature
formula $\Sigma_x$, and \eqref{eq:q_x_integral},
\begin{align*}
&\lambda(x)q_x(t)=\lambda(x)g(x)\leq\sum_{u\in
S_x}\lambda_x(u)g(u)=\int_{-1}^1g(s)w(s)ds=\\
&=\int_{-1}^1q_x\left(-1+\frac{1+t}{1+x}(1+s)\right)w(s)ds=\frac{1+x}{1+t}\int_{-1}^{1-2\frac{x-t}{1+x}}q_x(\sigma)w\left(-1+\frac{1+x}{
1+t}(1+\sigma)\right)d\sigma\leq\\
&\leq \frac{1+x}{1+t}\cdot\frac{M}{m}\int_{-1}^{1-2\frac{x-t}{1+x}}q_x(\sigma)w(\sigma)d\sigma\leq \frac{1+x}{1+t}\cdot\frac{M}{m}\int_{-1}^1 q_x(\sigma)w(\sigma)d\sigma=\frac{1+x}{1+t}\cdot\frac{M}{m}\lambda(x).\qedhere
\end{align*}
\end{proof}

\begin{proof}[Proof of Proposition~\ref{propos:q bound for t far from x}]
We divide into three cases.
\begin{enumerate}
\item Suppose $x=\xi_r(a)$ for some $1\leq r\leq n$ and $|a|\leq\frac{1}{(1-\sgn(a)t)\min\left\{n,\frac{1}{\sqrt{1-t^2}}\right\}}$.

Let $t_0=\eta_{r'}$ where $1\leq r'\leq n-1$ is such that
$|r'-r|\leq 1$. By Theorem \ref{thm:Badkov_thm} and Proposition
\ref{Gauss},
$$\lvert P_a(t_0)\rvert=\lvert\varphi(t_0)\rvert=\lvert\varphi(t_0)\rvert+\sqrt{1-t_0^2}\lvert\psi(t_0)\rvert\geq \frac{c(w)}{(1-t_0^2)^{1/4}}$$
and by Corollary~\ref{cor:Badkov},
\begin{align*}&|P_a(t)|\leq C(w)\max\left\{\min\left\{n,\frac{1}{\sqrt{1-t^2}}\right\}^{1/2},|a|(1-\sgn(a)t)\min\left\{n,\frac{1}{\sqrt{1-t^2}}\right\}^{3/2}\right\}=\\
&=C(w)\min\left\{n,\frac{1}{\sqrt{1-t^2}}\right\}^{1/2}.
\end{align*}
Therefore, by \eqref{eq:Lagrange_representation},
\begin{equation*}
\frac{q_x(t)}{q_x(t_0)}=\frac{1+\sgn(a)t}{1+\sgn(a)t_0}\left(\frac{P_a(t)(t_0-x)}{P_a(t_0)(t-x)}\right)^2\leq
C(w)\frac{(1+\sgn(a)t)\min\left\{n,\frac{1}{\sqrt{1-t^2}}\right\}(t_0-x)^2}{(1+\sgn(a)t_0)\frac{1}{\sqrt{1-t_0^2}}(t-x)^2}.
\end{equation*}
Hence, using Proposition \ref{Gauss}, Lemma \ref{lem:q naive bound} and Lemma \ref{"erdos-turan"} we get
\begin{align}
&q_x(t)=\frac{q_x(t)}{q_x(t_0)}q_x(t_0)\leq  C(w)\frac{(1+\sgn(a)t)\min\left\{n,\frac{1}{\sqrt{1-t^2}}\right\}(t_0-x)^2}{(1+\sgn(a)t_0)\frac{1}{\sqrt{1-t_0^2}}(t-x)^2}\leq\nonumber\\
&\leq
C(w)\frac{(1+\sgn(a)t)\max\left\{\frac{1}{n^2},1-\sgn(a)x\right\}\max\left\{\frac{1}{n},\sqrt{1-x^2}\right\}}{n\max\left\{1,n\sqrt{1-t^2}\right\}(t-x)^2}.\label{eq:q_x_t_first_case}
\end{align}
The bound \eqref{eq:q_bound_far1} follows immediately.

To obtain \eqref{eq:q_bound_far2} assume first that
$-1+\frac{1}{4n^2}\le x\le 1-\frac{1}{4n^2}$. By
\eqref{eq:q_x_t_first_case},
\begin{align*}
&q_x(t)\leq
\frac{C(w)}{\max\left\{1,n\sqrt{1-t^2}\right\}}\frac{(1+\sgn(a)t)(1-\sgn(a)x)}{|t-x|}\cdot\frac{1-x^2}{|t-x|}.
\end{align*}
We deduce that $q_x(t)\leq C(w)$ when
$t\notin(\frac{x-1}{2},\frac{x+1}{2})$ since then
$|t-x|\geq\frac{1}{4}(1+\sgn(a)t)(1-\sgn(a)x)$ and
$|t-x|\geq\frac{1-x^2}{4}$ . If $t\in(\frac{x-1}{2},\frac{x+1}{2})$
then $q_x(t)\leq C(w)$ by Lemma \ref{lem:q naive bound}.

Assume now that $1-\frac{1}{4n^2}<x<1$ (the case
$-1<x<-1+\frac{1}{4n^2}$ is treated similarly). Let $M$ be the
maximum of $q_x$ on $[-1,1]$ and suppose it is obtained in $t_1$. If
$t_1<\frac{x-1}{2}$, then by \eqref{eq:q_bound_far1},
\begin{equation*}M=q_x(t_1)\leq\frac{C(w)}{n\max\left\{1,n\sqrt{1-t_1^2}\right\}(t_1-x)^2}\leq\frac{C(w)}{n}\leq C(w).\end{equation*}
If $\frac{x-1}{2}\leq t_1\leq \frac{x+1}{2}$ then $M=q_x(t_1)\leq
C(w)$ by Lemma \ref{lem:q naive bound}. Finally, if
$t_1>\frac{x+1}{2}$ then, since $\lvert q_x'\rvert\leq(2n-1)^2M$
everywhere in $(-1,1)$ by Markov's inequality (as in
\eqref{eq:Bernstein_Markov_inequality}),
$$\left\lvert q_x\left(\frac{x+1}{2}\right)\right\rvert\geq\lvert q_x(t_1)\rvert-(2n-1)^2M\left(t_1-\frac{x+1}{2}\right)\geq M-(2n-1)^2M\frac{1-x}{2}\geq\frac{M}{2}$$
and since we have already seen that $\lvert
q_x(\frac{x+1}{2})\rvert\leq C(w)$ we get that $M\leq 2C(w)$.
\item Suppose $x=\xi_r(a)$ for some $1\leq r\leq n$ and $|a|>\frac{1}{(1-\sgn(a)t)\min\left\{n,\frac{1}{\sqrt{1-t^2}}\right\}}$.

Let $t_0=\xi_r(0)$. By Theorem \ref{thm:Badkov_thm} and Proposition
\ref{Gauss},
$$\lvert P_a(t_0)\rvert=|a|(1-\sgn(a)t_0)\lvert\psi(t_0)\rvert\geq \frac{c(w)|a|(1-\sgn(a)t_0)}{(1-t_0^2)^{3/4}}$$
and by Corollary~\ref{cor:Badkov},
\begin{align*}&|P_a(t)|\leq C(w)\max\left\{\min\left\{n,\frac{1}{\sqrt{1-t^2}}\right\}^{1/2},|a|(1-\sgn(a)t)\min\left\{n,\frac{1}{\sqrt{1-t^2}}\right\}^{3/2}\right\}=\\
&=C(w)|a|(1-\sgn(a)t)\min\left\{n,\frac{1}{\sqrt{1-t^2}}\right\}^{3/2}.
\end{align*}
Therefore, by \eqref{eq:Lagrange_representation},
\begin{equation*}\frac{q_x(t)}{q_x(t_0)}=\frac{1+\sgn(a)t}{1+\sgn(a)t_0}\left(\frac{P_a(t)(t_0-x)}{P_a(t_0)(t-x)}\right)^2\leq C(w)\frac{(1-\sgn(a)t)(1-t^2)\min\left\{n,\frac{1}{\sqrt{1-t^2}}\right\}^3(t_0-x)^2}{(1-\sgn(a)t_0)\frac{1}{\sqrt{1-t_0^2}}(t-x)^2}.
\end{equation*}
Hence, using Proposition \ref{Gauss}, Lemma \ref{lem:q naive bound} and Lemma \ref{"erdos-turan"} we get
\begin{align}
&q_x(t)=\frac{q_x(t)}{q_x(t_0)}q_x(t_0)\leq  C(w)\frac{(1-\sgn(a)t)(1-t^2)\min\left\{n,\frac{1}{\sqrt{1-t^2}}\right\}^3(t_0-x)^2}{(1-\sgn(a)t_0)\frac{1}{\sqrt{1-t_0^2}}(t-x)^2}\leq\nonumber\\
&\leq
C(w)\frac{(1-\sgn(a)t)(1+\sgn(a)x)\frac{1-t^2}{\max\left\{\frac{1}{n^2},1-t^2\right\}}\max\left\{\frac{1}{n},\sqrt{1-x^2}\right\}}{n\max\left\{1,n\sqrt{1-t^2}\right\}(t-x)^2}.\label{eq:q_x_t_second_case}
\end{align}
The bound \eqref{eq:q_bound_far1} follows immediately.

To obtain \eqref{eq:q_bound_far2}, first note that if
$t\in(\frac{x-1}{2},\frac{x+1}{2})$ then $q_x(t)\le C(w)$ by
Lemma~\ref{lem:q naive bound}. Second note that if
$t\notin(\frac{x-1}{2},\frac{x+1}{2})$ then
$|t-x|\geq\frac{1}{4}(1-\sgn(a)t)(1+\sgn(a)x)$,
$|t-x|\geq\frac{1-x^2}{4}$ and $|t-x|\geq\frac{1-t^2}{2}$. Thus,
using \eqref{eq:q_x_t_second_case},
\begin{align*}
q_x(t)\leq
C(w)\frac{n\sqrt{1-t^2}\max\left\{\frac{1}{n},\sqrt{1-x^2}\right\}}{\max\left\{1,n\sqrt{1-t^2}\right\}\sqrt{|t-x|}}\le
C(w)\frac{\max\left\{\sqrt{1-t^2},\sqrt{1-x^2}\right\}}{\sqrt{|t-x|}}\le
C(w).
\end{align*}

\item Suppose $x=\eta_r$ for some $1\leq r\leq n-1$.

This case follows from case 2 above by noting that for every
$-1<t<1$,
\begin{equation*}
  q_{\eta_r}(t) =
  \frac{1-\eta_r}{1-t}\lim_{a\to\infty}q_{\xi_r(a)}(t) =
  \frac{1+\eta_r}{1+t}\lim_{a\to-\infty}q_{\xi_{r+1}(a)}(t).\qedhere
\end{equation*}
\end{enumerate}
\end{proof}

\section{Proof of Theorem~\ref{thm:Lipschitz} and Theorem~\ref{thm:abs_cont}}\label{sec:pi_derivative}

In this section we prove Theorem~\ref{thm:Lipschitz} and Theorem~\ref{thm:abs_cont}. We first explain how the proposition follows from several lemmas. The proof of these
lemmas is delayed to the next subsections.

Recall the definition \eqref{eq:index_fcn_def} of the index
function. In a similar manner to the definition \eqref{eq:q_x_def}
of the polynomial $q_x$, we define, for each $-1<x<1$, the
polynomial $p_{x}$ to be the unique polynomial satisfying the
following properties:
\begin{align}
  &\deg(p_x)\le\sum_{u\in S_x} I(x) - 2,\\
  &p_x(u)=\begin{cases} 1& u\in S_{x}\cap[-1,x]\\
  0&u\in
S_{x}\cap(x,1]
  \end{cases},\label{eq:p_x_value_at_roots}\\
  &p'_{x}(u)=0,\quad u\in
(S_{x}-\{x\})\cap(-1,1)\label{eq:p_x_prime_prop}.
\end{align}

\begin{figure}[t]
\centering
{\includegraphics[width=0.7\textwidth]{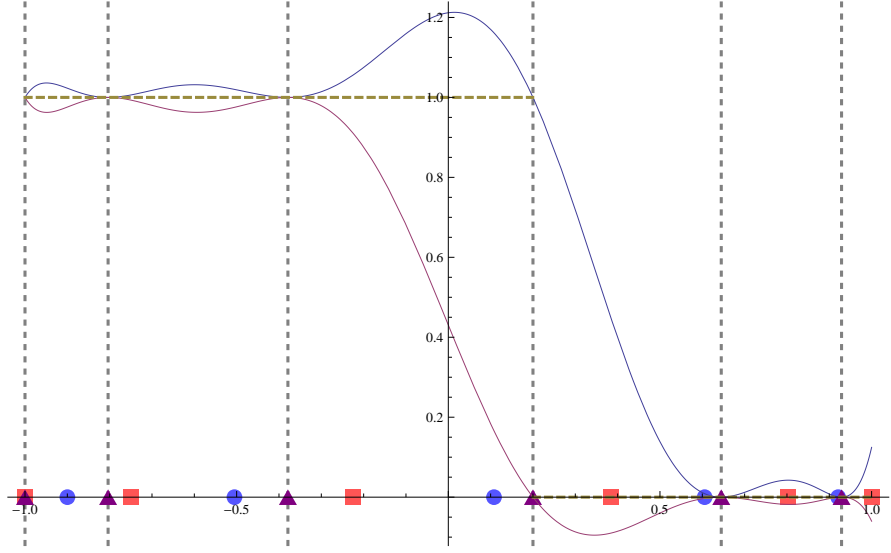}}
{\caption{A plot of the polynomials $p_x$ and $\underline{p}_x$ for
$n=5$, $x=0.2$ and the weight function $w(t)=\max\{1,1+4t\}$.
Observe that $p_x$ lies above and $\underline{p}_x$ lies below the
indicator function of the interval $[0,x]$, as proved in
Lemma~\ref{>0}. The circles on the axis denote the nodes of the
Gaussian quadrature, the squares denote the nodes of the Lobatto
quadrature and the triangles denote the nodes of the quadrature
formula $\Sigma_{0.2}$.\label{fig:p_x_and_p_underline_x_poly}}}
\end{figure}
We note that $\deg(p_x)=\sum_{u\in S_x} I(x) - 2$ unless $x\ge
\xi_n(0)$, in which case $p_x\equiv 1$.
Figure~\ref{fig:p_x_and_p_underline_x_poly} shows the graph of this
polynomial as well as the polynomial $\underline{p}_x$ defined
below. Using the quadrature formula $\Sigma_x$ and the fact that
$\deg(p_x)\le 2n-1$ it follows immediately that
\begin{equation*}
  \int_{-1}^1 p_x(t) w(t)dt = \pi(x).
\end{equation*}

Our first lemma relates the quantity $\pi'$, which we would like to
estimate, to the polynomial $p_x$.

\begin{lemma}\label{deriv-shift}
For every weight function $w$ and every differentiability point
$x\in(-1,1)$ of $\pi$ we have
\begin{equation}\label{eq:pi_minus_estimate}
\pi'(x)=-\lambda(x) p'_x(x) =
\lambda_{x}(-1)p'_{x}(-1)+\lambda_{x}(1)p'_{x}(1)-\int_{-1}^1p_{x}'(t)w(t)dt.
\end{equation}
\end{lemma}

Our next lemmas estimate the terms on the right-hand side of
\eqref{eq:pi_minus_estimate}.

\begin{lemma}\label{lem:pre-prop3.3} Suppose $w$ satisfies the conditions of
Theorem~\ref{thm:Badkov_thm} with $\alpha=0$. Then there exists a
constant $C(w)>0$ such that for every $-1<x<\xi_n(0)$,
\begin{equation}\label{eq:lem7.2_lower_bound}\lambda_{x}(-1)p'_{x}(-1)+\lambda_{x}(1)p'_{x}(1)\geq p_{x}(1)w(1)-C(w)\frac{\lambda(x)}{1-x},\end{equation}
and for every $-1< x<1$,
\begin{equation}\label{eq:lem7.2_upper_bound}\lambda_{x}(-1)p'_{x}(-1)+\lambda_{x}(1)p'_{x}(1)\leq-\left(p_x(-1)-1\right)w(-1)+C(w)\lambda(x)\min\left\{\frac{1}{1+x},n^2\right\}.\end{equation}
\end{lemma}

\begin{lemma}\label{step}
Suppose $w$ is an absolutely continuous weight function. For every
$-1<x<1$,
\begin{equation*}-\left(p_x(-1)-1\right)w(-1)-\int_{-1}^1 q_x(t)\bigl\lvert w'(t)\bigr\rvert
  dt\leq\int_{-1}^1p_{x}'(t)w(t)dt + w(x)\leq p_{x}(1)w(1)+\int_{-1}^1 q_x(t)\bigl\lvert w'(t)\bigr\rvert
  dt.\end{equation*}
\end{lemma}

\begin{lemma}\label{one more step}
\begin{enumerate}
\item  Suppose $w$ satisfies the assumptions of Theorem~\ref{thm:Lipschitz}, and let $R$ and $m$ be the constants from these assumptions.
Then for every $-1<x<1$,
$$\int_{-1}^1 q_x(t)\bigl\lvert w'(t)\bigr\rvert
  dt\leq\frac{R}{m}\lambda(x).$$
\item Suppose $w$ satisfies the assumptions of Theorem~\ref{thm:abs_cont},
and let $p$ be the constant from these assumptions. Then for every
$-1<x<1$,
$$\int_{-1}^1 q_x(t)\bigl\lvert w'(t)\bigr\rvert
  dt\leq C(w)\lambda(x)^{1-\frac{1}{p}}.$$
\end{enumerate}
\end{lemma}

We are now prepared to prove Theorem~\ref{thm:Lipschitz} and Theorem~\ref{thm:abs_cont}.
\begin{proof}[Proof of Theorem~\ref{thm:Lipschitz} and Theorem~\ref{thm:abs_cont}]
We prove only the lower bounds. The proof of the upper bounds is similar and slightly
simpler.

When $x\ge \xi_n(0)$ the bounds follow by taking $C(w)$ large
enough, using the fact that $\pi$ is non-decreasing by
\eqref{eq:pi_underline_pi_lambda_def} so that $\pi'(x)\ge 0$, using
Corollary \ref{cor:lambda_lower} and using Proposition~\ref{Gauss}
to see that $\frac{\lambda(x)}{1-x}\geq c(w)$. Combining
Lemmas~\ref{deriv-shift}, \ref{lem:pre-prop3.3} and \ref{step},
$$\pi'(x)-w(x)\geq -C(w)\frac{\lambda(x)}{1-x}-\int_{-1}^1 q_x(t)\bigl\lvert w'(t)\bigr\rvert
  dt$$
for every differentiability point $-1<x<\xi_n(0)$ of $\pi$. The
stated lower bounds now follow from Lemma \ref{one more step}.
\end{proof}

In the next subsections we prove the above lemmas. To this aim we
introduce a second polynomial $\underline{p}_x$, whose properties we
now explain (see Figure~\ref{fig:p_x_and_p_underline_x_poly}).

We define, for each $-1<x<1$, the polynomial $\underline{p}_{x}$ to
be the unique polynomial satisfying the following properties:
\begin{align}
  &\deg(\underline{p}_x)\le\sum_{u\in S_x} I(x) - 2,\\
  &\underline{p}_x(u)=\begin{cases} 1& u\in S_{x}\cap[-1,x)\\
  0&u\in
S_{x}\cap[x,1]
  \end{cases},\label{eq:underline_p_x_value_at_roots}\\
  &\underline{p}'_{x}(u)=0,\quad u\in
(S_{x}-\{x\})\cap(-1,1)\label{eq:underline_p_x_prime_prop}.
\end{align}
Here, again, $\deg(\underline{p}_x) = \sum_{u\in S_x} I(x) - 2$
unless $x\le\xi_1(0)$, in which case $\underline{p}_x\equiv 0$.

Observe that $p_{x}-\underline{p}_{x}$ is a polynomial of degree
$\leq \sum_{u\in S_x} I(x) - 2$ satisfying
$(p_{x}-\underline{p}_{x})(x)=1$, $(p_{x}-\underline{p}_{x})(u)=0$
for every $u\in S_{x}-\{x\}$, and $(p_{x}-\underline{p}_{x})'(u)=0$
for every $u\in (S_{x}-\{x\})\cap(-1,1)$. Comparing with the
properties of $q_x$ following \eqref{eq:q_x_def} we conclude that
(see also Figures~\ref{fig:q_x} and
\ref{fig:p_x_and_p_underline_x_poly})
\begin{equation}\label{eq:q_as_difference}
p_{x}-\underline{p}_{x}=q_{x}.
\end{equation}

We write $\chi_A$ for the characteristic function of the set $A$.
\begin{lemma}\label{>0} Let $w$ be any weight function and let $-1<x<1$. Then
\begin{equation*}
  \underline{p}_{x}(t)\leq\chi_{[-1,x)}(t)\leq\chi_{[-1,x]}(t)\leq p_{x}(t)\quad -1\leq t\leq
  1.
\end{equation*}
In addition,
\begin{equation*}
\begin{aligned}
  &\underline{p}_{x}'(-1)\le 0\le p_{x}'(-1)&&\text{when $-1\in S_x$},\\
  &p_{x}'(-1)\le 0\le \underline{p}_{x}'(-1)&&\text{when $-1\notin S_x$},\\
  &p_{x}'(1)\le 0\le \underline{p}_{x}'(1)&&\text{when $1\in S_x$},\\
  &\underline{p}_{x}'(1)\le 0\le p_{x}'(1)&&\text{when $1\notin S_x$}.\\
\end{aligned}
\end{equation*}
\end{lemma}

\begin{proof}
If $p_x\equiv 1$, i.e., $x\ge \xi_n(0)$, the claims relating to it
are trivial. Otherwise, by Rolle's theorem, $p'_{x}$ vanishes at
some point (strictly) between any two consecutive points of
$S_x\cap[-1,x]$ and any two consecutive points of $S_{x}\cap(x,1]$.
Together with the points of $(S_{x}-\{x\})\cap(-1,1)$ we obtain
$\sum_{u\in S_x} I(x) - 3$ distinct points in which $p'_{x}$
vanishes. Since $p'_{x}$ is a polynomial of degree $\sum_{u\in S_x}
I(x) - 3$ we conclude that these are all the points in which it
vanishes, and that it changes sign in each of them and in no other
point. The statements concerning $p_{x}$ now follow since, by
definition, $p_{x}(x)=1>0$.

The statements concerning $\underline{p}_{x}$ follow either by using
a similar argument, or by noting that $\underline{p}_{x}
(t)=1-\tilde{p}_{-x}(-t)$ where $\tilde{p}$ is the polynomial $p$
defined with respect to the reversed weight function
$\tilde{w}(t):=w(-t)$.
\end{proof}

\subsection{Proof of Lemma~\ref{deriv-shift}}
Fix $-1<x<1$ to be a differentiability point of $\pi$ (recall from
Section~\ref{sec:background} that $\pi$ is differentiable at all but
finitely many points of $(-1,1)$). For the proof of the lemma, we
generalize the definition of the polynomial $p_x$ to a one-parameter
family of polynomials $p_x(y,\cdot)$. Let $U\subset(-1,1)$ be an
open interval containing $x$ and not containing any other node of
the quadrature formula $\Sigma_x$. For every $y\in U$ we let
$p_x(y,\cdot)$ be the unique polynomial satisfying the following
properties:
\begin{align}
  &\deg(p_x(y,\cdot))\le\sum_{u\in S_x} I(x) - 2,\nonumber\\
  &p_x(y,u)=\begin{cases} 1& u\in (S_x-\{x\}))\cap[-1,y)\\
  1&u=y\\
  0&(S_{x}-\{x\})\cap(y,1]
  \end{cases},\label{eq:p_x_u_def}\\
  &p'_{x}(y,u)=0,\quad u\in
(S_{x}-\{x\})\cap(-1,1)\nonumber.
\end{align}

\begin{figure}[t]
\centering
{\includegraphics[width=0.7\textwidth]{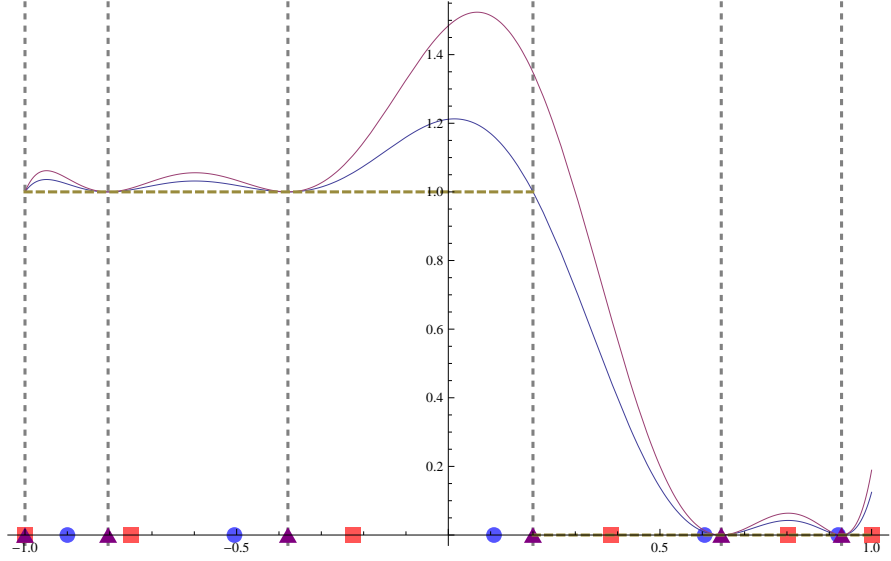}}
{\caption{A plot of the polynomials $p_x(\cdot)$ and $p_x(y,\cdot)$
for $n=5$, $x=0.2$, $y=0.3$ and the weight function
$w(t)=\max\{1,1+4t\}$. The circles on the axis denote the nodes of
the Gaussian quadrature, the squares denote the nodes of the Lobatto
quadrature and the triangles denote the nodes of the quadrature
formula $\Sigma_{0.2}$.\label{fig:p_x_and_p_x_y_poly}}}
\end{figure}
As for $p_x$, $\deg(p_x(y,\cdot))=\sum_{u\in S_x} I(x) - 2$ unless
$x\ge\xi_n(0)$, in which case $p_x(y,\cdot)\equiv 1$. The graph of
this polynomial is shown in Figure~\ref{fig:p_x_and_p_x_y_poly}.
With this definition, $p_x = p_x(x,\cdot)$. Clearly, $\deg
p_x(y,\cdot)\leq 2n-1$. In addition, it is not difficult to see that
$p_x(y,\cdot)\geq \chi_{[-1,y]}$ in $[-1,1]$ in the same manner as
in the proof of Lemma~\ref{>0}. Thus, applying the quadrature
formula $\Sigma_y$,
\begin{equation}\label{eq:p_x(y,t)}\int_{-1}^1
p_x(y,t)w(t)dt=\sum_{u\in S_y}\lambda_y(u)p_x(y,u)\geq\sum_{u\in
S_y\cap[-1,y]}\lambda_y(u)=\pi(y),\quad y\in U
\end{equation}
with equality when $y=x$. We claim that it follows that
\begin{equation}\label{eq:pi_prime_expr}
\pi'(x)=\left(\frac{d}{dy}\int_{-1}^1
p_x(y,t)w(t)dt\right)\Big|_{y=x}=\int_{-1}^1 \frac{\partial
p_x}{\partial y}(x,t)w(t)dt.
\end{equation}
To see this, observe first that $p_x(y,t)$ is a continuous rational
function of $y$ and $t$ in the rectangle $U\times[-1,1]$. Thus, the
validity of the differentiation under the integral sign follows from
the bounded convergence theorem. Second, note that the first
equality in \eqref{eq:pi_prime_expr} follows from the fact that
\eqref{eq:p_x(y,t)} holds in an \emph{open} neighborhood of $x$,
with equality at $x$.

Since $p_x(y,\cdot)$ is a polynomial of degree at most $2n-1$, it
follows that also $\frac{\partial p_x}{\partial y}(x,\cdot)$ is a
polynomial of degree at most $2n-1$, so we can use the quadrature
formula $\Sigma_{x}$ to calculate
\begin{equation}\label{eq:quadrature_for_p_x_deriv}
\int_{-1}^1\frac{\partial p_x}{\partial y}(x,t)w(t)dt=\sum_{u\in
S_{x}}\lambda_{x}(u)\frac{\partial p_x}{\partial y}(x,u).
\end{equation}
Now observe that for each $u\in S_x-\{x\}$, the definition
\eqref{eq:p_x_u_def} implies that $p_x(y,u)$ is constant when $y\in
U$. Thus all terms involving $u\neq x$ in the right-hand side of
\eqref{eq:quadrature_for_p_x_deriv} vanish. Using
\eqref{eq:pi_prime_expr}, we conclude that
\begin{equation}\label{eq:pi_prime_x_intermediate_expr}
  \pi'(x) = \lambda(x) \frac{\partial p_x}{\partial y}(x,x).
\end{equation}
Now note that, since $p_x(z,z)$ is identically 1 by
\eqref{eq:p_x_u_def}, then by the chain rule:
$$0=\frac{d p_x}{dz}(z,z)=\frac{\partial p_x}{\partial y}(z,z)+\frac{\partial p_x}{\partial u}(z,z).$$
In particular,
$$\frac{\partial p_x}{\partial y}(x,x)=-\frac{\partial p_x}{\partial u}(x,x)=-p_x'(x).$$
Together with \eqref{eq:pi_prime_x_intermediate_expr} this yields
the first equality in the statement of the lemma.

To obtain the second equality of the lemma, note that, since $\deg
p'_{x}\leq 2n-2$, we may use $\Sigma_{x}$ and
\eqref{eq:p_x_prime_prop} to obtain
$$\int_{-1}^1 p_x'(t)w(t)dt=\sum_{u\in S_{x}}\lambda_{x}(u)p_{x}'(u)=\lambda_{x}(-1)p'_{x}(-1)+\lambda(x)p'_{x}(x)+\lambda_{x}(1)p'_{x}(1).$$

\subsection{Proof of
Lemma~\ref{lem:pre-prop3.3}}\label{sec:proof_of_lemma_pre-prop3.3}
Fix a weight function $w$ satisfying the conditions of
Theorem~\ref{thm:Badkov_thm} with $\alpha=0$. In the following
claims we examine more closely the behaviour of the polynomial $p_x$
at the endpoints of the interval.
\begin{claim}\label{cl:p_endpoints} There exists a
constant $C(w)>0$ such that if $x=\xi_r(a)$ then
\begin{align}
p_{x}(1)&\leq C(w)\frac{\lambda(x)}{1-x},\quad\text{$1\leq r\leq
n-1$\label{eq:p_x_1_upper_bound_1}
and $0\leq a<\infty$},\\[5pt]
p_{x}(-1)-1&\leq C(w)\lambda(x)\min\left\{\frac{1}{1 +x},
n^2\right\},\quad\text{$1\leq r\leq n$ and $-\infty<a\leq
0$}.\label{eq:p_x_1_upper_bound_2}
\end{align}
\end{claim}

\begin{proof}
By Lemma \ref{P_a},
$$|P'_a(x)|\geq c(w)\sqrt{\frac{n}{\lambda(x)(1-x^2)}},\quad x\in[\xi_1(0),\xi_n(0)].$$
Here, one may obtain some improvement to the following bounds when
$|a|$ is sufficiently large by using the full bound given by
Lemma~\ref{P_a}. However, these improvements do not seem to carry
over to small values of $|a|$. In addition, by the upper bound in
Corollary~\ref{cor:Badkov},
\begin{equation*}
\begin{aligned}
  |P_a(1)|&=|\varphi(1)|\leq C(w)\sqrt{n},&&0\leq a<\infty,\\
  |P_a(-1)|&=|\varphi(-1)|\leq C(w)\sqrt{n},&&-\infty<a\leq 0.
\end{aligned}
\end{equation*}
Thus \eqref{eq:p_x_1_upper_bound_1} follows using Lemma \ref{>0} and
\eqref{eq:Lagrange_representation} since, for $a>0$,
\begin{equation*}
  p_{x}(1)=q_{x}(1)+\underline{p}_{x}(1)\leq q_{x}(1)=\frac{2}{x+1}\left(\frac{P_a(1)}{(1-x)P_a'(x)}\right)^2\leq
  C(w)\frac{\lambda(x)}{1-x},
\end{equation*}
and similarly, if $a=0$,
\begin{equation}\label{eq:p_x_1_upper_bound_a_0}
  p_{x}(1)\leq q_{x}(1)=\left(\frac{P_a(1)}{(1-x)P_a'(x)}\right)^2\leq
  C(w)\frac{(1+x)\lambda(x)}{1-x}.
\end{equation}
In a similar manner, we obtain \eqref{eq:p_x_1_upper_bound_2} when
$2\le r\le n$ or when $r=1$ and $a=0$ since
\begin{equation}\label{eq:p_x_minus_1_estimate}
  p_{x}(-1)-1\leq q_{x}(-1)=\left(\frac{2}{1-x}\right)^{\chi_{(-\infty,0)}(a)}\left(\frac{P_a(-1)}{(1+x)P_a'(x)}\right)^2
  \leq
  C(w)\frac{\lambda(x)}{1+x}
\end{equation}
and since $\frac{1}{1+x}\le C(w) n^2$ by Proposition~\ref{Gauss}. It
remains to prove \eqref{eq:p_x_1_upper_bound_2} when $r=1$ and
$a<0$. For this case, by Lemma \ref{P_a} and
Proposition~\ref{Gauss},
\begin{equation*}
|P'_a(x)|\geq
c(w)\sqrt{\frac{n}{\lambda(x)}}\max\left\{\frac{|a|}{1-x},\frac{1}{\sqrt{1-\xi_1(0)^2}}\right\}\ge
c(w)\sqrt{\frac{n}{\lambda(x)}}\max\left\{|a|,n\right\}.
\end{equation*}
Thus, as in \eqref{eq:p_x_minus_1_estimate},
\begin{equation}\label{eq:p_x_minus_1_estimate_edge}
  p_{x}(-1)-1\leq q_{x}(-1)=\left(\frac{2}{1-x}\right)\left(\frac{P_a(-1)}{(1+x)P_a'(x)}\right)^2
  \leq C(w)\frac{\lambda(x)}{(1+x)^2}\min\left\{\frac{1}{|a|^2},\frac{1}{n^2}\right\}.
\end{equation}
Fix $\eps(w)>0$, small enough for the following calculation. We
consider separately two cases. First suppose that $x\le -1 +
\frac{\eps(w)}{n^2}$. By \eqref{eq:P_a_def} we have $|a|(1+x) =
|\varphi(x)/\psi(x)|$. Hence, using \eqref{eq:psi_upper_bound} and
\eqref{eq:phi_psi_lower_bound_new}, if $\eps(w)$ is sufficiently
small then
\begin{equation*}
  |a|(1+x) \ge \frac{c(w)}{n}.
\end{equation*}
Plugging this into \eqref{eq:p_x_minus_1_estimate_edge} proves
\eqref{eq:p_x_1_upper_bound_2} in this case. Now suppose that $x\ge
-1 + \frac{\eps(w)}{n^2}$. Here, \eqref{eq:p_x_1_upper_bound_2}
follows directly from \eqref{eq:p_x_minus_1_estimate_edge}.
\end{proof}

\begin{claim}\label{cl:p'_endpoints} There exists a
constant $C(w)>0$ such that if $x=\xi_r(a)$ for $1\leq r\leq n$ then
\begin{align}
-p'_{x}(1)&\leq C(w)n^2\frac{\lambda(x)}{1-x},\quad -\infty<a<0,\label{eq:p'_x_1_lower_bound}\\[5pt]
p'_{x}(-1)&\leq C(w)n^2\frac{\lambda(x)}{1+x},\quad
0<a<\infty\label{eq:p'_x_-1_upper_bound}.
\end{align}
\end{claim}

\begin{proof}
By the upper bound in Corollary~\ref{cor:Badkov},
$$|P_a(\pm 1)|\leq C(w)\left(1+|a|n\right)\sqrt{n},\quad -\infty<a<\infty.$$
Thus, assuming that $-\infty<a<0$, Lemma \ref{>0},
\eqref{eq:Lagrange_representation} and Lemma~\ref{P_a} yield
\begin{align}
-p'_{x}(1)&=-q'_{x}(1)-\underline{p}'_{x}(1)\leq
-q'_{x}(1)=\frac{1}{1-x}\left(\frac{P_a(1)}{(1-x)P_a'(x)}\right)^2\le C(w)\frac{n(1+(-a)n)^2}{(1-x)^3\left(P_a'(x)\right)^2}\le\nonumber\\
&\le C(w)\frac{(1+(-a)n)^2 \lambda(x)}{(1-x)^3}
\min\left\{\frac{(1-x)^2}{a^2},
1-\overline{x}^2\right\}\le\nonumber\\&\le C(w)\frac{(1+(-a)n)^2
\lambda(x)}{(1-x)^2} \min\left\{\frac{1-x}{a^2}, 1+\overline{x}\right\}\le\nonumber\\
&\le C(w)\frac{\lambda(x)}{(1-x)^2}\cdot
\begin{cases}
  1+\overline{x}&|a|\le \frac{1}{n}\\
  a^2n^2(1+\overline{x})&\frac{1}{n}<|a|\le \sqrt{\frac{1-x}{1+\overline{x}}}\\
  n^2(1-x)&|a|>\sqrt{\frac{1-x}{1+\overline{x}}}
\end{cases},\label{eq:p'_1_bound}
\end{align}
where we recall the definition of $\overline{x}$ from
\eqref{eq:x_bar_def}. The bound \eqref{eq:p'_x_1_lower_bound} now
follows with the aid of Proposition~\ref{Gauss}.

In a similar manner, assuming that $0<a<\infty$, we have
\begin{align*}
p'_{x}(-1)&\leq
q'_{x}(-1)=\frac{1}{1+x}\left(\frac{P_a(-1)}{(1+x)P_a'(x)}\right)^2\le
C(w)\frac{(1+a n)^2 \lambda(x)}{(1+x)^2} \min\left\{\frac{1+x}{a^2},
1-\overline{x}\right\},
\end{align*}
yielding the bound \eqref{eq:p'_x_-1_upper_bound}.
\end{proof}

\begin{claim}\label{cl:p'_endpoints_a=-infty} There exists a
constant $C(w)>0$ such that if $x=\eta_r$ for $1\leq r\leq n-1$ then
\begin{align}
-p'_{x}(1)&\leq C(w)n^2\frac{\lambda(x)}{1-x},\label{eq:p'_1_eta_bound}\\
p'_{x}(-1)&\leq
C(w)n^2\frac{\lambda(x)}{1+x}\label{eq:p'_-1_eta_bound}.
\end{align}
\end{claim}
\begin{proof} By the upper bound in Corollary~\ref{cor:Badkov},
$$|\psi(\pm1)|\leq C(w)n\sqrt{n}.$$
Thus, \eqref{eq:p'_1_eta_bound} follows using Lemma \ref{>0},
\eqref{eq:Lagrange_representation} and Lemma~\ref{P_a}, by
\begin{align}
-p'_{x}(1)&=-q'_{x}(1)-\underline{p}'_{x}(1)\le
-q'_{x}(1)=\frac{2}{1-x^2}\left(\frac{\psi(1)}{(1-x)\psi'(x)}\right)^2\le\nonumber \\
&\le C(w)\frac{n^3}{(1-x^2)(1-x)^2\left(\psi'(x)\right)^2}\le
C(w)n^2 \lambda(x)\frac{1+x}{1-x}\leq
C(w)n^2\frac{\lambda(x)}{1-x}.\label{eq:p_prime_x_1_upper_bound_calc}
\end{align}
In a similar manner, \eqref{eq:p'_-1_eta_bound} follows by
\begin{equation*}
p'_{x}(-1)\le
q'_{x}(-1)=\frac{2}{1-x^2}\left(\frac{\psi(-1)}{(1+x)\psi'(x)}\right)^2\leq
C(w)n^2\frac{\lambda(x)}{1+x}.\qedhere
\end{equation*}
\end{proof}

\begin{proof}[Proof of Lemma \ref{lem:pre-prop3.3}] We first prove \eqref{eq:lem7.2_lower_bound}. We consider separately three cases.
\begin{enumerate}
\item Suppose $x=\xi_r(a)$ for $0\leq a<\infty$, $1\leq r\leq
  n-1$. In this case, the claim follows since $p_{x}(1)\leq C(w)\frac{\lambda(x)}{1-x}$ by
  \eqref{eq:p_x_1_upper_bound_1}, $\lambda_{x}(1)=0$ since $1\notin
  S_x$ and $\lambda_x(-1)p'_{x}(-1)\ge 0$ since $\lambda_x(-1)=0$ if $a=0$
  and $p'_x(-1)\ge 0$ if $a>0$ by Lemma \ref{>0}.
\item Suppose $x=\xi_r(a)$ for $-\infty< a<0$, $1\leq r\leq n$. In this
case, the claim follows since $p_{x}(1)=0$ by
\eqref{eq:p_x_value_at_roots}, $\lambda_{x}(-1)=0$ since $-1\notin
S_x$, $\lambda_{x}(1)\leq C(w)\cdot\frac{1}{n^2}$ by Lemma
\ref{lambda1} and $-p'_{x}(1)\leq C(w)n^2\frac{\lambda(x)}{1-x}$ by
\eqref{eq:p'_x_1_lower_bound}.
\item Suppose $x=\eta_r$ for $1\leq r\leq n-1$. In this case, the claim follows since
$p_{x}(1)=0$, $p'_{x}(-1)\ge0$, $\lambda_{x}(1)\leq
C(w)\cdot\frac{1}{n^2}$ by Lemma \ref{lambda1} and $-p'_{x}(1)\leq
C(w)n^2\frac{\lambda(x)}{1-x}$ by \eqref{eq:p'_1_eta_bound}.
\end{enumerate}
We now prove \eqref{eq:lem7.2_upper_bound}. Again we consider
separately three cases. We appeal to Proposition~\ref{Gauss} to
justify that $\frac{1}{1+x}\le C(w)n^2$ in the first and third
cases.
\begin{enumerate}
\item Suppose $x=\xi_r(a)$ for $0<a<\infty$, $1\leq r\leq
  n$. In this case, the claim follows since $p_{x}(-1)=1$ by
\eqref{eq:p_x_value_at_roots}, $\lambda_{x}(1)=0$ since $1\notin
S_x$, $\lambda_{x}(-1)\leq C(w)\cdot\frac{1}{n^2}$ by Lemma
\ref{lambda1} and $p'_{x}(-1)\leq C(w)n^2\frac{\lambda(x)}{1+x}$ by
\eqref{eq:p'_x_-1_upper_bound}.
\item Suppose $x=\xi_r(a)$ for $-\infty< a\leq 0$, $1\leq r\leq n$. In this
case, the claim follows since $p_{x}(-1)-1\leq
C(w)\lambda(x)\min\left\{\frac{1}{1+x},n^2\right\}$ by
\eqref{eq:p_x_1_upper_bound_2}, $\lambda_{x}(-1)=0$ since $-1\notin
S_x$ and $\lambda_x(1)p'_{x}(1)\le 0$ since $\lambda_x(1)=0$ if
$a=0$ and $p'_x(1)\le 0$ if $a<0$ by Lemma \ref{>0}.
\item Suppose $x=\eta_r$ for $1\leq r\leq n-1$. In this case, the claim follows since
$p_{x}(-1)=1$, $p'_{x}(1)\le 0$, $\lambda_{x}(-1)\leq
C(w)\cdot\frac{1}{n^2}$ by Lemma \ref{lambda1} and $p'_{x}(-1)\leq
C(w)n^2\frac{\lambda(x)}{1+x}$ by
\eqref{eq:p'_-1_eta_bound}.\qedhere
\end{enumerate}
\end{proof}

\subsection{Proof of Lemma~\ref{step}}\label{sec:proof_of_lemma_step}

Integrating by parts, which is possible since $w$ is
absolutely continuous, we get
\begin{align*}
  &\int_{-1}^1p_x'(t)w(t)dt=\Bigl[p_x(1)w(1)-p_x(-1)w(-1)\Bigr]-\int_{-1}^1p_x(t)w'(t)dt=\\
  &=\Bigl[p_x(1)w(1)-p_x(-1)w(-1)\Bigr]-\int_{-1}^1\chi_{[-1,x]}(t)w'(t)dt-\int_{-1}^1\left(p_x-\chi_{[-1,x]}\right)(t)w'(t)dt=\\
&=\Bigl[p_x(1)w(1)-p_x(-1)w(-1)\Bigr]-\Bigl[w(x)-w(-1)\Bigr]-\int_{-1}^1\left(p_x-\chi_{[-1,x]}\right)(t)w'(t)dt=\\
  &=-w(x)+\Bigl[p_x(1)w(1)-\left(p_x(-1)-1\right)w(-1)\Bigr]-\int_{-1}^1\left(p_x-\chi_{[-1,x]}\right)(t)w'(t)dt.
\end{align*}
The lemma now follows since $p_x(1)\geq 0$ and $p_x(-1)\geq1$ by
Lemma \ref{>0}, and
\begin{align*}
  \biggl\lvert&\int_{-1}^1\left(p_{x}-\chi_{[-1,x]}\right)(t)w'(t)dt\biggr\rvert\leq\int_{-1}^1\left(p_{x}-\chi_{[-1,x]}\right)(t)\bigl\lvert w'(t)\bigr\rvert
  dt\leq\\
&\leq\int_{-1}^1(p_{x}-\underline{p}_x)(t)\bigl\lvert
w'(t)\bigr\rvert
  dt=\int_{-1}^1 q_x(t)\bigl\lvert w'(t)\bigr\rvert
  dt
\end{align*}
by another application of Lemma~\ref{>0} and
\eqref{eq:q_as_difference}.

\subsection{Proof of
Lemma~\ref{one more step}}\label{sec:proof_of_lemma_one_more_step}
If $w$ satisfies the assumptions of Theorem~\ref{thm:Lipschitz} then using \eqref{eq:q_x_integral} we get
\begin{equation*}
\int_{-1}^1 q_x(t)\bigl\lvert w'(t)\bigr\rvert
  dt\leq\frac{R}{m}\int_{-1}^1 q_x(t)w(t)dt=\frac{R}{m}\lambda(x).
\end{equation*}
If $w$ satisfies the assumptions of Theorem~\ref{thm:abs_cont} and
$w'\in L_p[-1,1]$ for some $p>1$, then using Holder's inequality,
the first part of Proposition \ref{propos:q bound for t far from x},
and \eqref{eq:q_x_integral},
\begin{align*}
&\int_{-1}^1 q_x(t)\bigl\lvert w'(t)\bigr\rvert
  dt\leq\lVert q_x\rVert_{\frac{p}{p-1}}\lVert w'\rVert_{p}\leq\left(\max_{-1\leq t\leq 1}q_x(t)\right)^{\frac{1}{p}}\left(\int_{-1}^1 q_x(t)dt\right)^{1-\frac{1}{p}}\lVert w'\rVert_{p}\leq\\
&\leq C(w)\left(\frac{1}{m}\int_{-1}^1
q_x(t)w(t)dt\right)^{1-\frac{1}{p}}\lVert
w'\rVert_{p}=C(w)\left(\frac{\lambda(x)}{m}\right)^{1-\frac{1}{p}}\lVert
w'\rVert_{p}\le C(w)\lambda(x)^{1-\frac{1}{p}}.
\end{align*}

\section{Discontinuous weights}\label{sec:thm_main2_proof}
In this section we prove Theorem~\ref{thm:discont}. The theorem
follows as an immediate consequence, using Lemma~\ref{lambda}, from
the following proposition.

\begin{propos}\label{prop:discont}
Suppose $w$ is a weight function on $[-1,1]$ satisfying the assumptions of Theorem
\ref{thm:discont}.
For every $\eps>0$ there exists an $n_0=n_0(w,\eps)$ such that if $n\ge
n_0$ then for every differentiability point $x\in(-1,1)$ of $\pi$,
\begin{align}
-C(w)\frac{\lambda(x)}{1-x}-\eps-C(w)\sum_{i=1}^L
\min\left\{\frac{1}{n^2(s_i-x)^2},1\right\}\leq\pi'(x)&-w(x)\leq\nonumber\\
\leq C(w)\lambda(x)\min\left\{\frac{1}{1+x},n^2\right\}+\eps&+C(w)\sum_{i=1}^L
\min\left\{\frac{1}{n^2(s_i-x)^2},1\right\}.\label{eq:required_pi_minus_estimate2_2}
\end{align}
\end{propos}
Figures~\ref{fig:pi_integral_and_pi_underline_discont},
\ref{fig:pi_derivative_discont} and \ref{fig:lambda_discont} show
the graphs of $\pi$, $\pi'-w$ and $\lambda$ for a discontinuous
weight function $w$ satisfying the assumptions of Theorem
\ref{thm:discont}.

\begin{figure}[t]
\centering
{\includegraphics[width=0.7\textwidth]{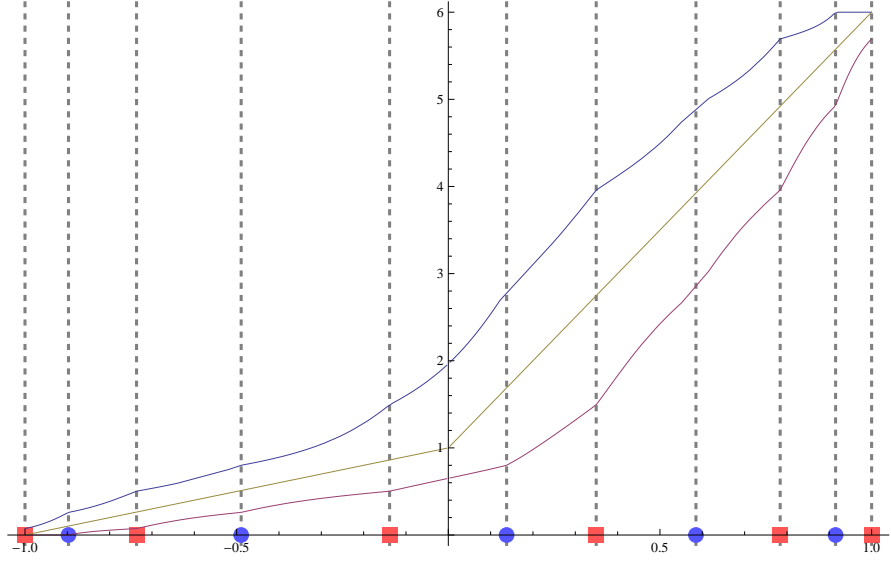}}
{\caption{A plot of $\pi$ (top graph), $\int_{-1}^x w(t)dt$ (middle
graph) and $\underline{\pi}$ (bottom graph) for $n=5$ and the weight
function $w$ defined by $w(t)=1$ if $t<0$ and $w(t)=5$ if $t\ge 0$.
Observe that $\pi$ lies above the graph of the integral, as the
Chebyshev-Markov-Stieltjes inequalities guarantee (see
\eqref{eq:Chebyshev_Markov_Stieltjes_intro}). The circles on the
axis denote the nodes of the Gaussian quadrature and the squares
denote the nodes of the Lobatto
quadrature.\label{fig:pi_integral_and_pi_underline_discont}}}
\end{figure}

\begin{figure}[t]
\centering
{\includegraphics[width=0.7\textwidth]{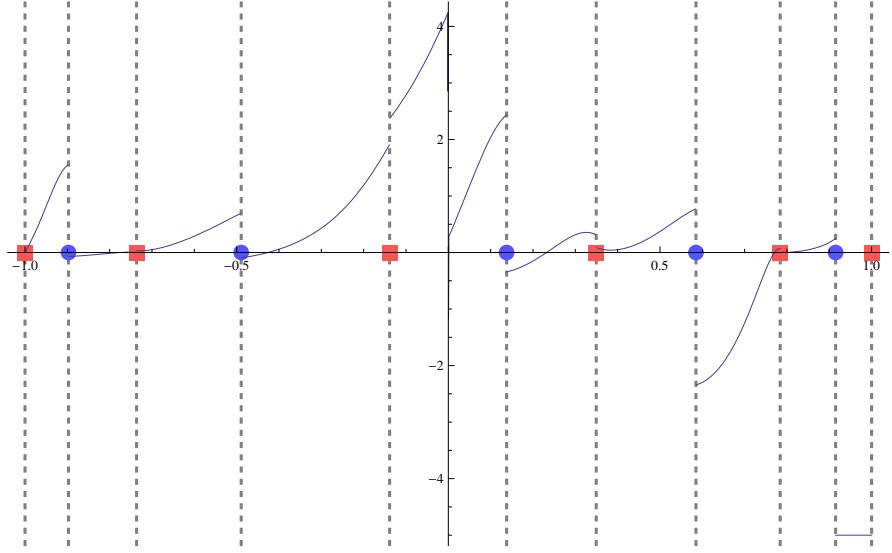}}
{\caption{A plot of the function $\pi' - w$ for $n=5$ and the weight
function $w$ defined by $w(t)=1$ if $t<0$ and $w(t)=5$ if $t\ge 0$.
The circles on the axis denote the nodes of the Gaussian quadrature
and the squares denote the nodes of the Lobatto quadrature. We note
that the jump at $t=0$ is exactly the jump of $w$ at this point as
we know that $\pi$ is analytic there, see
Section~\ref{sec:background}. \label{fig:pi_derivative_discont}}}
\end{figure}

\begin{figure}[t]
\centering {\includegraphics[width=0.7\textwidth]{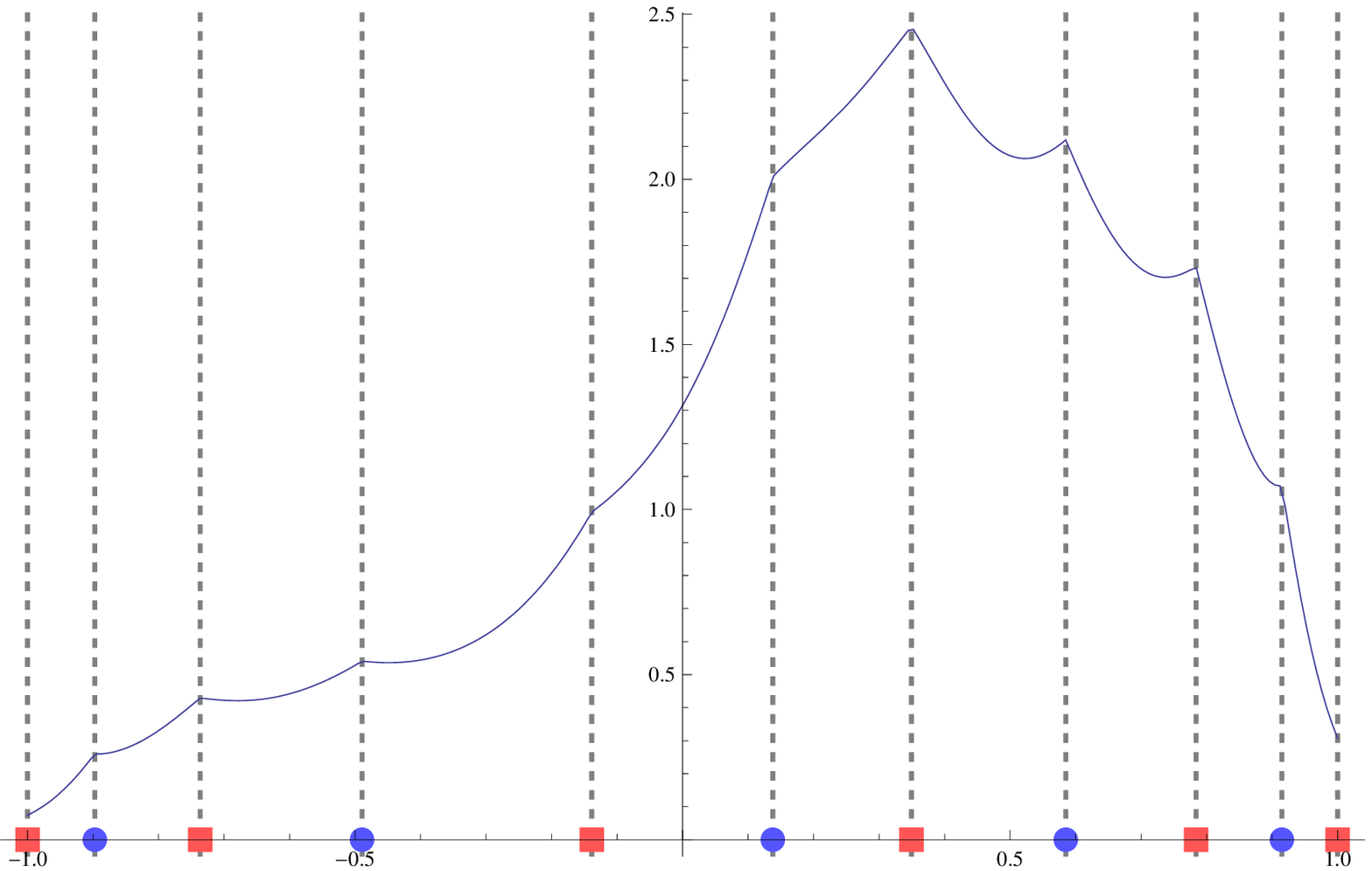}}
{\caption{A plot of the function $\lambda$ for $n=5$ and the weight
function $w$ defined by $w(t)=1$ if $t<0$ and $w(t)=5$ if $t\ge 0$.
The circles on the axis denote the nodes of the Gaussian quadrature
and the squares denote the nodes of the Lobatto
quadrature.\label{fig:lambda_discont}}}
\end{figure}

The proof of Proposition~\ref{prop:discont} follows the same
strategy as that of Theorem~\ref{thm:abs_cont}. Indeed, all the
ingredients used in the proof of Theorem~\ref{thm:abs_cont}, with
the exceptions of Lemma~\ref{step} and Lemma~\ref{one more step},
are proved for weight functions satisfying the assumptions of
Theorem~\ref{thm:Badkov_thm} with $\alpha=0$ and are thus valid also
for weight functions satisfying the assumptions of
Theorem~\ref{thm:discont}. We prove only the lower bound in
Proposition~\ref{prop:discont} as the proof of the upper bound is
similar. A replacement for Lemmas~\ref{step} and~\ref{one more step}
is provided by the next two lemmas.

We recall that by the assumptions of Theorem~\ref{thm:discont}, $w$
has left and right limits at each point $s$, which will be denoted
$w(s-)$ and $w(s+)$, respectively. We remind the reader of the
definition of $q_x$ from Section~\ref{sec:p_a_derivative} and the
definition of $p_x$ from Section~\ref{sec:pi_derivative}.
\begin{lemma}\label{step2} Suppose $w$ satisfies the assumptions of Theorem~\ref{thm:discont}. Then for every $-1<x<1$, $x\notin\{s_1,\ldots, s_L\}$,
\begin{equation}\label{eq:step2_replacement}
\int_{-1}^1p_x'(t)w(t)dt\leq-w(x)+p_x(1)w(1)-\sum_{i=1}^L\left(w(s_i+)-w(s_i-)\right)\left(p_x-\chi_{[-1,x]}\right)(s_i)+\int_{-1}^1 q_x(t)\cdot\bigl\lvert w'(t)\bigr\rvert
  dt.
\end{equation}
\end{lemma}
\begin{proof}
Using integration by parts on the interval $[-1,s_1]$, which is
possible since $w$ is absolutely continuous on $[-1,s_1]$ when
interpreting $w(s_1)$ as $w(s_1-)$, we have
\begin{multline*}
\int_{-1}^{s_1}p_x'(t)w(t)dt=\Bigl[p_x(s_1)w(s_1-)-p_x(-1)w(-1)\Bigr]-\int_{-1}^{s_1}p_x(t)w'(t)dt=\\
=\Bigl[p_x(s_1)w(s_1-)-p_x(-1)w(-1)\Bigr]-\int_{-1}^{s_1}\chi_{[-1,x]}(t)w'(t)dt-\int_{-1}^{s_1}\Bigl(p_x-\chi_{[-1,x]}\Bigr)(t)w'(t)dt.
\end{multline*}
 Similarly,
\begin{multline*}
\int_{s_L}^1p_x'(t)w(t)dt=\Bigl[p_x(1)w(1)-p_x(s_L)w(s_L+)\Bigr]-\int_{s_L}^1p_x(t)w'(t)dt=\\
=\Bigl[p_x(1)w(1)-p_x(s_L)w(s_L+)\Bigr]-\int_{s_L}^1\chi_{[-1,x]}(t)w'(t)dt-\int_{s_L}^1\Bigl(p_x-\chi_{[-1,x]}\Bigr)(t)w'(t)dt
\end{multline*}
and, for every $2\leq i\leq L$,
\begin{multline*}
\int_{s_{i-1}}^{s_i}p_x'(t)w(t)dt=\Bigl[p_x(s_i)w(s_i-)-p_x(s_{i-1})w(s_{i-1}+)\Bigr]-\int_{s_{i-1}}^{s_i}p_x(t)w'(t)dt=\\
=\Bigl[p_x(s_i)w(s_i-)-p_x(s_{i-1})w(s_{i-1}+)\Bigr]-\int_{s_{i-1}}^{s_i}\chi_{[-1,x]}(t)w'(t)dt-\int_{s_{i-1}}^{s_i}\Bigl(p_x-\chi_{[-1,x]}\Bigr)(t)w'(t)dt.
\end{multline*}
Summing these $L+1$ equalities and noticing that
\begin{align*}
\int_{-1}^1 \chi_{[-1,x]}(t)w'(t)dt&=w(x)-w(-1) -\sum_{i=1}^L
(w(s_i+)-w(s_i-))\chi_{[-1,x]}(s_i)
\end{align*}
we get
\begin{align*}
\int_{-1}^1p_x'(t)w(t)dt=&-w(x)+\left(1-p_x(-1)\right)w(-1)+p_x(1)w(1)-\\
&-\sum_{i=1}^L\left(w(s_i+)-w(s_i-)\right)\left(p_x-\chi_{[-1,x]}\right)(s_i)-\int_{-1}^1\bigl(p_x-\chi_{[-1,x]}\bigr)(t)w'(t)dt.
\end{align*}
The lemma now follows since $p_x(-1)\geq1$ by
Lemma \ref{>0}, and
\begin{align*}
  \biggl\lvert&\int_{-1}^1\left(p_{x}-\chi_{[-1,x]}\right)(t)w'(t)dt\biggr\rvert\leq\int_{-1}^1\left(p_{x}-\chi_{[-1,x]}\right)(t)\bigl\lvert w'(t)\bigr\rvert
  dt\leq\\
&\leq\int_{-1}^1(p_{x}-\underline{p}_x)(t)\bigl\lvert
w'(t)\bigr\rvert
  dt=\int_{-1}^1 q_x(t)\bigl\lvert w'(t)\bigr\rvert
  dt
\end{align*}
by another application of Lemma~\ref{>0} and
\eqref{eq:q_as_difference}.
\end{proof}

\begin{lemma}\label{one more step2}
Suppose $w$ satisfies the assumptions of Theorem~\ref{thm:discont}.
For every $\eps>0$ there exists an $n_0=n_0(w,\eps)$ such that if
$n\ge n_0$ then for every $-1<x<1$,
$$\int_{-1}^1 q_x(t)\bigl\lvert w'(t)\bigr\rvert
  dt\leq \eps.$$
\end{lemma}
\begin{proof}
Let $\tilde{\eps}>0$. If $w$ satisfies the assumptions of
Theorem~\ref{thm:discont} then there is a
$\delta=\delta(w,\tilde{\eps})>0$ such that $\int_I\lvert
w'(t)\rvert dt<\tilde{\eps}$ for every interval $I\subseteq[-1,1]$
such that $|I|\le 2\delta$. Thus using both statements of
Proposition \ref{propos:q bound for t far from x},
\begin{align*}
&\int_{-1}^1 q_x(t)\bigl\lvert w'(t)\bigr\rvert
  dt=\int_{\{-1\leq t\leq 1:|t-x|\leq\delta\}}q_x(t)\bigl\lvert w'(t)\bigr\rvert
  dt+\int_{\{-1\leq t\leq 1: |t-x|>\delta\}}q_x(t)\bigl\lvert w'(t)\bigr\rvert
  dt\leq\\
&\leq C(w)\int_{\{-1\leq t\leq 1: |t-x|\leq\delta\}}\bigl\lvert w'(t)\bigr\rvert
  dt+\frac{C(w)}{n\delta^2}\int_{\{-1\leq t\leq 1: |t-x|>\delta\}}\bigl\lvert w'(t)\bigr\rvert
  dt\leq C(w)\tilde{\eps}+\frac{C(w)}{n\delta^2}\le C(w)\tilde{\eps}\end{align*}
when $n\ge \frac{1}{\tilde{\eps}\delta(w,\tilde{\eps})^2}$. The
stated bound follows by choosing $\tilde{\eps} = c(w)\eps$.
\end{proof}
We now prove Proposition~\ref{prop:discont}.
\begin{proof}[Proof of the lower bound in Proposition~\ref{prop:discont}]
When $x\ge \xi_n(0)$ the bound follows by taking $C(w)$ large
enough, using the fact that $\pi$ is non-decreasing by
\eqref{eq:pi_underline_pi_lambda_def} so that $\pi'(x)\ge 0$, using
Corollary \ref{cor:lambda_lower} and using Proposition~\ref{Gauss}
to see that $\frac{\lambda(x)}{1-x}\geq c(w)$.

For $x<\xi_n(0)$, combining Lemmas~\ref{deriv-shift},
\ref{lem:pre-prop3.3} and \ref{step2} yields
\begin{equation*}
  \pi'(x)\geq w(x) - C(w)\frac{\lambda(x)}{1-x} +\sum_{i=1}^L
  \left(w(s_i+)-w(s_i-)\right)\left(p_x-\chi_{[-1,x]}\right)(s_i) -
  \int_{-1}^1 q_x(t)\cdot\bigl\lvert w'(t)\bigr\rvert
  dt.
\end{equation*}
By Lemma \ref{>0}, \eqref{eq:q_as_difference} and Proposition
\ref{propos:q bound for t far from x}, for each $i$,
\begin{multline*}
  \left|w(s_i+)-w(s_i-)\right|\left(p_x-\chi_{[-1,x]}\right)(s_i)\le
  C(w)q_x(s_i)\le\\
  \le C(w)\min\left\{\frac{1}{n^2\sqrt{1-s_i^2}\,(s_i-x)^2},1\right\}\le
  C(w)\min\left\{\frac{1}{n^2(s_i-x)^2},1\right\}.
\end{multline*}
The proposition follows using Lemma \ref{one more step2}.
\end{proof}
\begin{rem}\label{rem:refined_discontinuity} As can be seen from the proof, for every $1\leq i\leq L$, the term $\min\left\{\frac{1}{n^2(s_i-x)^2},1\right\}$ may be omitted, either from the lower bound, if $w(s_i+)\geq w(s_i-)$, or from the upper bound, if $w(s_i+)\leq w(s_i-)$.
\end{rem}
\section{Discussion and open problems}\label{sec:open_problems}

The main result in our work is a differential version of the
Chebyshev-Markov-Stieltjes inequalities given by
Theorem~\ref{thm:Lipschitz}, Theorem~\ref{thm:abs_cont} and
Theorem~\ref{thm:discont}. The Chebyshev-Markov-Stieltjes
inequalities hold for every weight function (and more generally, any
measure). In what generality does a differential version of the
inequalities hold? Our results show that some version holds for all
weight functions which are absolutely continuous and bounded away
from zero, and also for a certain class of discontinuous weight
functions. To what extent are such assumptions on the weight
function necessary for the result? Do similar results hold for
Jacobi weight functions, when $w(x) = (1-x)^{\alpha}(1+x)^{\beta}$?

In addition, what are the best possible error terms in
Theorem~\ref{thm:Lipschitz}? Writing $x=\xi_i(a)$, these error terms
may be improved for certain ranges of $i$ and $a$, see
Section~\ref{sec:proof_of_lemma_pre-prop3.3}, e.g., in the proof of
Claim~\ref{cl:p_endpoints} and in estimates
\eqref{eq:p_x_1_upper_bound_a_0}, \eqref{eq:p'_1_bound} and
\eqref{eq:p_prime_x_1_upper_bound_calc}, but it is not clear what
would be the form of the sharp bounds. This question may be asked
also for weight functions satisfying the assumptions of
Theorem~\ref{thm:abs_cont} or Theorem~\ref{thm:discont}.

To make the bounds in our theorems fully effective one would require
explicit bounds for the constants $C(w)$ appearing in them. When $w$
satisfies the conditions of Theorem~\ref{thm:Lipschitz} quantitative
estimates for $C(w)$ in terms of the Lipschitz constant and minimal
value of $w$ may be obtained from our proof (including the proof of
Theorem~\ref{thm:Badkov_thm} for this case in
Appendix~\ref{sec:appendix_Badkov}, where one may obtain
quantitative estimates for $\ell_{n}$ and $f_n$ via the Korous
comparison principle and \cite[(11.3.6)]{Szego}). We do not know to
similarly bound the constants appearing in
Theorem~\ref{thm:abs_cont} and Theorem~\ref{thm:discont} by
parameters depending only on the minimal value and the regularity of
$w$. This is due to the fact that the dependence on $w$ in the
constants appearing in Theorem~\ref{thm:Badkov_thm} is non-explicit.

\section*{Acknowledgements}
We thank Vladimir Badkov, Percy Deift, Eli Levin, Doron Lubinsky,
Paul Nevai and Mikhail Sodin for helpful remarks and discussions
during the course of this work.

\appendix
\section{Appendix: Remarks on Badkov's theorem}\label{sec:appendix_Badkov}
In this section we provide some remarks on Badkov's
Theorem~\ref{thm:Badkov_thm} including a proof for the case that the
function $h$ is Lipschitz continuous and $\alpha=0$, the main case
in our proof of Theorem~\ref{thm:Lipschitz}.

Write $(\varphi_n)$ and $(\psi_n)$ for the orthogonal polynomials on
$[-1,1]$ with respect to the weight functions $w(t)$ and
$(1-t^2)w(t)$, respectively, so that in the notation of our paper,
$\varphi = \varphi_n$ and $\psi = \psi_{n-1}$. The starting point
for the theorem is a relation between $(\varphi_n), (\psi_n)$ and
orthogonal polynomials on the unit circle for a related weight
function. Define $\tilde{w}:[-\pi,\pi]\to[0,\infty)$ by
\begin{equation}\label{eq:tilde_w_def}
  \tilde{w}(\theta):=w(\cos \theta)|\sin \theta|.
\end{equation}
Let $(\phi_n)$ be the orthogonal polynomials on the unit circle with
respect to $\tilde{w}$, that is, $\deg(\phi_n)=n$ and
\begin{equation*}
  \frac{1}{2\pi}\int_{-\pi}^\pi
  \phi_n(e^{i\theta})\phi_m(e^{i\theta})\tilde{w}(\theta)d\theta =
  \delta_{n,m},
\end{equation*}
normalized to have real positive leading coefficients, which we
denote by $\ell_n$. We also let $f_n:= \phi_n(0)$ be the constant
term of $\phi_n$. The following relation, a consequence of
\cite[Theorem 11.5]{Szego}, connects the three systems of orthogonal
polynomials,
\begin{equation}\label{eq:three_orthogonal_systems}
\sqrt{\frac{2}{\pi}}e^{-in\theta}\phi_{2n}(e^{i\theta})=\sqrt{1+\frac{f_{2n}}{\ell_{2n}}}\varphi_n(\cos
\theta)+i\sqrt{1-\frac{f_{2n}}{\ell_{2n}}}\sin\theta\,\psi_{n-1}(\cos\theta),
\quad -\pi\le \theta\le \pi,\,n\ge 1.
\end{equation}
The next lemma uses this relation to show that
Theorem~\ref{thm:Badkov_thm} is equivalent to estimating
$|\phi_{2n}|$ on the unit circle.
\begin{lemma}\label{lem:unit_circle_polynomials}
  If
  \begin{equation}\label{eq:Szego_condition}
    \int_{-\pi}^\pi |\log \tilde{w}(\theta)|d\theta <\infty
  \end{equation}
  then there exist constants $C(w),c(w)>0$ such that for every $n\ge
  1$,
  \begin{equation*}
    c(w)|\phi_{2n}(e^{i\theta})|\le |\varphi_n(x)|+\sqrt{1-x^2}|\psi_{n-1}(x)|\le
    C(w)|\phi_{2n}(e^{i\theta})|, \quad x = \cos\theta,\,-\pi\le \theta\le \pi.
  \end{equation*}
\end{lemma}
\begin{proof}
Equation 11.3.12 in \cite{Szego} implies that $\ell_n$ tends to a
  positive limit, whence equation 11.3.6 in \cite{Szego} implies
  that $f_n$ tends to zero. The lemma follows from these facts using equation
  \eqref{eq:three_orthogonal_systems}.
\end{proof}
Theorems 1.2 and 1.4 of Badkov \cite{Badkov} give two-sided
estimates on $|\phi_n|$ on the unit circle under rather general
assumptions on $\tilde{w}$ which include the assumptions of
Theorem~\ref{thm:Badkov_thm}. When the weight $w$ is assumed to be
Lipschitz continuous such estimates may also be derived by means of
the Korous comparison theorem. We proceed to describe this
derivation for the case $\alpha=0$ which we are interested in (see
also \cite{Badkov_older}) and comment briefly on the more general
case at the end of the appendix.

Assume now that $w$ satisfies the assumptions of
Theorem~\ref{thm:Lipschitz}. The proof will follow by a comparison
argument. Denote by $u$ the constant weight function, $u\equiv 1$ on
$[-1,1]$. Let $(L_n)$ be the Legendre polynomials, defined in
\eqref{eq:Legendre_poly_def}, orthogonal with respect to $u$. The
polynomials $(L'_n)$ are orthogonal with respect to the weight
$1-t^2$ and we have the normalizations \cite[(4.21.7),
(4.3.3)]{Szego}
$$\int_{-1}^1 L_n^2(t)dt=\frac{1}{n+\frac{1}{2}}\quad\text{and}\quad\int_{-1}^1 {L'_n}^2(t)(1-t^2)dt=\frac{n(n+1)}{n+\frac{1}{2}}.$$
Define also the orthonormal versions,
$$\overline{L_n}=\sqrt{n+\frac{1}{2}}L_n\quad\text{and}\quad\overline{L'_n}=\sqrt{\frac{n+\frac{1}{2}}{n(n+1)}}L'_n.$$
\begin{lemma}\label{lem:Legendre_poly_bounds}
  There exist absolute constants $C,c>0$ such that for every $-1<x<1$ and $n\ge 1$,
  \begin{equation}\label{ineq:Legendre}
    c\min\left\{n,\frac{1}{\sqrt{1-x^2}}\right\}^{1/2}\leq|\overline{L_n}(x)|+\sqrt{1-x^2}|\overline{L'_n}(x)|\leq C\min\left\{n,\frac{1}{\sqrt{1-x^2}}\right\}^{1/2}.
  \end{equation}
\end{lemma}
\begin{proof}
  The upper bound follows by standard estimates of Jacobi
  polynomials \cite[Theorem 7.32.2]{Szego}. For the lower bound, define $F_n(x):=n(n+1)L_n^2(x)+(1-x^2)L_n'^2(x)$. It suffices to
  show that
  \begin{equation}\label{eq:F_n_lower_bound}
    F_n(x) \ge c\,
    n\min\left\{\frac{1}{\sqrt{1-x^2}},n\right\}.
  \end{equation}
  This estimate holds by \cite[Theorem 8.21.13]{Szego} when $n\ge
  n_0$ and $1-x^2 \ge \frac{c'}{n^2}$, for some $n_0, c'>0$. It
  holds trivially for $n<n_0$, adjusting the constant $c$ as
  necessary, since $L_n$ has no double root. Finally,
  \eqref{eq:F_n_lower_bound} follows also when $n\ge n_0$ and
  $1-x^2\le \frac{c'}{n^2}$ by observing that the differential equation for the Legendre polynomials implies that $F_n$ is monotone decreasing on [-1,0] and monotone increasing on
  [0,1], see \cite[(7.3.4)]{Szego}.
\end{proof}
Since $w$ satisfies the assumptions of  Theorem~\ref{thm:Lipschitz} we may apply the Korous comparison theorem \cite[Theorem
7.1.3]{Szego} to obtain
\begin{align}
|\varphi_n(x)|&\leq
C(w)\left(|\overline{L_{n-1}}(x)|+|\overline{L_n}(x)\right),&|\psi_{n-1}(x)|\leq
C(w)\big(|\overline{L'_{n-1}}(x)|+|\overline{L'_n}(x)|\big),\label{ineq:varphi_by_L}\\
|\overline{L_n}(x)|&\leq
C(w)\left(|\varphi_{n-1}(x)|+|\varphi_n(x)|\right),&|\overline{L'_n}(x)|\leq
C(w)\left(|\psi_{n-2}(x)|+|\psi_{n-1}(x)|\right).\label{ineq:psi_by_L}
\end{align}
The upper bound in \eqref{eq:Badkov_bound} (for $\alpha=0$) now
follows by combining the inequalities in \eqref{ineq:varphi_by_L}
and using Lemma~\ref{lem:Legendre_poly_bounds}. To obtain the lower
bound note first that by \cite[(11.4.6)]{Szego} we have
\begin{equation}|\phi_n(z)|=\frac{|\ell_{n+1}\phi_{n+1}(z)-f_{n+1}z^{n+1}\overline{\phi_{n+1}(z)}|}{|\ell_n z|}\le C(w)|\phi_{n+1}(z)|,\quad |z|=1,\, n\ge
0\label{ineq:phi_n_by_phi_n+1},
\end{equation}
where we have used that $\ell_n\to\ell>0$ and $f_n\to 0$ as in the
proof of Lemma~\ref{lem:unit_circle_polynomials}. Finally, the lower
bound follows by combining the inequalities in
\eqref{ineq:psi_by_L}, using Lemma~\ref{lem:Legendre_poly_bounds},
applying Lemma~\ref{lem:unit_circle_polynomials} and using
\eqref{ineq:phi_n_by_phi_n+1} twice.

We finish by briefly remarking on the method used in Badkov's paper
\cite{Badkov} from which the general case of
Theorem~\ref{thm:Badkov_thm} follows. Badkov begins by upper
bounding $|\phi_n|$ on the unit circle via the so-called Szeg\H o
function $\pi$ associated with the weight $\tilde{w}$. This bound is
up to an error $(1 + \delta_n \sqrt{n})$ for a related quantity
$\delta_n$ \cite[Lemma 4.2]{Badkov} (see also \cite[Theorem 3.6 and
Theorem 4.10]{Geronimus_book}). The advantage of such a bound is
that both the Szeg\H o function and the quantity $\delta_n$ are
multiplicative in the weight function $\tilde{w}$ \cite[Lemma
4.1]{Badkov}, thus allowing one to bound them separately for the
factors $|\sin \theta|$ and $w(\cos \theta)$ present in
\eqref{eq:tilde_w_def}. This task is undertaken in \cite[Theorem 2.1
and Theorem 4.1]{Badkov}. Condition~\eqref{eq:Badkov's condition} is
used for estimating $\delta_n$ for the factor $w(\cos \theta)$ via
results in \cite[Section 3.7]{Geronimus_book}. This provides the
required upper bound on $|\phi_n|$. To obtain the lower bound, the
function $|\phi_n' \phi_n|$ is estimated from below by the
Christoffel-Darboux kernel \cite[Lemma 11.1]{Badkov}. This kernel is
then estimated from below \cite[Theorem 9.2]{Badkov} and an upper
bound for $|\phi'_n|$ is derived from the upper bound for
$|\phi_n|$.

\end{document}